\documentclass[12pt]{article}
\usepackage{graphicx} 
\usepackage{amsmath,amssymb,amsthm,mathrsfs}
\usepackage{tikz}
\usepackage{latexsym}
\usepackage[left=25mm, right=25mm]{geometry}
\usepackage{comment}

\theoremstyle{plain}
\newtheorem{theorem}{Theorem}[section]
\newtheorem{proposition}[theorem]{Proposition}

\newtheorem{lemma}[theorem]{Lemma}
\newtheorem{example}[theorem]{Example}
\newtheorem{remark}[theorem]{Remark}
\newtheorem{definition}[theorem]{Definition}
\numberwithin{equation}{section}

\newcommand{\Z}{\mathbb{Z}}
\renewcommand{\O}{\mathcal{O}}
\renewcommand{\P}{\mathbb{P}}
\newcommand{\Q}{\mathbb{Q}}
\newcommand{\R}{\mathbb{R}}
\newcommand{\C}{\mathbb{C}}
\newcommand{\ii}{\sqrt{-1}}

\def\ba{\boldsymbol{a}}
\def\bb{\boldsymbol{b}}
\def\bc{\boldsymbol{c}}
\def\be{\boldsymbol{e}}
\def\bg{\boldsymbol{\gamma}}

\def\bp{\boldsymbol{p}}
\def\bq{\boldsymbol{q}}
\def\bu{\boldsymbol{u}}
\def\bv{\boldsymbol{v}}
\def\bw{\boldsymbol{w}}
\def\bx{\boldsymbol{x}}
\def\by{\boldsymbol{y}}
\def\bz{\boldsymbol{z}}
\def\balpha{\boldsymbol{\alpha}}
\def\bbeta{\boldsymbol{\beta}}
\def\wh#1{\widehat{#1}}
\def\ol#1{\overline{#1}}

\def\cF{\mathscr{F}}

\def\dR{{\mathrm{d\hspace{-0.2pt}R}}}            
\def\Hom{{\mathrm{Hom}}}  
\def\Fer{{\rm Fer}}  
\def\Fereq{F_d}  
\def\Spec{\operatorname{Spec}}     
\def\Sol{{\text{\bf Sol}}}

\def\Image{{\mathrm{Im}}}        
\def\Res{\mathrm{Res}}
\def\ord{{\mathrm{ord}}}
\def\Fan{{\mathrm{Fan}}}            
            
\def\vol{{\mathrm{vol}}}

\def\PG{R_{\Gamma,\psi}}           
\def\RNAfK{\pi_1^{\rm loc}(\mathrm{Fer})}

\def\ot{\otimes}
\def\lra{\longrightarrow}

\def\bt{{\boldsymbol{t}}}

\def\os#1#2{\overset{#1}{#2}}

\usepackage[affil-it]{authblk}

\title{
Periods of Limiting Mixed Hodge Structures of Projective Hypersurfaces
}
\author{Masanori Asakura and Saiei-Jaeyeong Matsubara-Heo}
\date{}

\begin{document}

\maketitle

\begin{abstract}
For a generic one-parameter degeneration of projective hypersurfaces, we show that the periods of the limiting mixed Hodge structure are generated by certain special values of 
logarithm, Gamma and Dirichlet $L$-functions.
Our proof is based on the analytic continuation of solutions to the GKZ system.
\end{abstract}

\section{Introduction}
The subject of this paper is 
the mixed Hodge structures arising from a one-parameter degeneration of algebraic varieties,
which we call the {\it limiting mixed Hodge structure}
(abbreviated as limiting MHS).
The limiting MHS was introduced and developed by
Schmid \cite{schmid1973variation} and Steenbrink \cite{steenbrink1976limits}.
Later on, in the theory of mixed Hodge modules by M. Saito \cite{saito1990mixed}, 
it was generalized via nearby cycle functors in full generality.
Roughly speaking, it is a theory of
asymptotic expansions of period integrals.
The periods have a link to number theory.
For example, Deligne introduces the so-called Deligne periods for
projective smooth varieties over $\Q$, and his conjecture asserts that 
they agree with critical values of motivic $L$-functions (\cite{deligne1979valeurs},
\cite[\S 2]{schneider4introduction}).
Furthermore, the Beilinson conjectures assert that for
mixed Hodge structures
arising from the higher Chern class maps in algebraic $K$-theory,
the off-diagonal entries 
of the period matrices provide
non-critical $L$-values (\cite[\S 5]{schneider4introduction}).
Although a period arising in a limiting MHS is not a period integral of an algebraic variety in the strict sense, 
Ayoub provided a motivic construction of nearby cycle functors in terms of Voevodsky's category
(cf. \cite{ayoub2007motivic}).
Thus, it is natural to expect that the limiting MHS is also connected to number theory.
In this paper, we focus on an arithmetic nature of the periods of the limiting MHS.

The limiting MHS is extensively studied for a family of hypersurfaces, and
the recent progress is brought from tropical geometry.
For example, E. Katz and Stapledon provide a combinatorial expression of
the Hodge numbers of the limiting MHSs of sch\"on toric hypersurfaces (\cite{katz2012tropical}, \cite{katz2016tropical}).
In contrast, the study of their periods is more limited.
Thanks to the influential work of Candelas et al. \cite{candelas1991pair},
the period of the limiting MHS often appears in the study of mirror symmetry (cf. \cite{vanStraten}).
We do not intend to give an exhaustive bibliography; instead, 
we only refer to 
his recent paper by Yamamoto \cite{yamamoto2024period}.
There, he studies the mirror symmetric gamma conjecture (cf. \cite{abouzaid2020gamma}),
and he shows that Riemann zeta values appear in the periods of the limiting MHS of 
a family of tropical  hypersurfaces.
It is noteworthy that the hypersurfaces discussed in \cite{yamamoto2024period}
are not necessarily Calabi-Yau manifolds, while 
the families are restricted to satisfy a certain strict condition.
The readers can find a list of related works in 
\cite[Introduction]{yamamoto2024period}.

Despite this progress, it is a widely open question to describe periods of the limiting MHS explicitly by arithmetic invariants.
The purpose of this paper is to give an answer 
to the question for generic degenerations of projective hypersurfaces.
To be more precise, we determine the asymptotic expansions of the period integrals of hypersurfaces completely, and as a consequence, 
we show that
the periods of limiting MHS are generated by special values of
logarithmic function, the Gamma function $\Gamma(s)$
and Dirichlet $L$-functions $L(s,\chi)$.
A caveat is that Riemann's zeta values are not enough to describe the periods of limiting MHS in contrast to the study of mirror symmetry (see e.g., \cite[Conjectures 1, 2]{vanStraten}).

To state our main theorem, let $n\geq 2$ be an integer
and let $(x_1,\ldots,x_n)$ denote the homogeneous
coordinates of $\P^{n-1}$.
For an integer $m>n$, we consider a generically smooth 
family of hypersurfaces $X$ defined by a homogeneous polynomial 
\[
F=\sum_{i=1}^m z_i\bx^{\ba_i}
\]
of degree $d$ with deformation parameters $z_i$.
Let $f_1(t),\ldots, f_m(t) \in K\{\!\{t\}\!\}[t^{-1}]$ be
non-zero meromorphic functions at $t=0$ whose Taylor expansions are
defined over a subfield $K\subset\C$.
Let $X_f/\Delta$ be a flat family of hypersurfaces defined by a polynomial $F_f=\sum_{i=1}^m f_i(t)\bx^{\ba_i}$ over the unit
disk $\Delta$.
Suppose that $X_f$ is smooth over $\Delta\setminus\{0\}$.
The {limiting MHS} for the family $X_f/\Delta$ is a quadruple
\[
H^{n-2}_\infty(X_f/\Delta)_K=(H_{\dR,\infty,K},\,H_{B,\infty},\,F^\bullet H_{\dR,\infty,K},\,W_\bullet H_{B,\infty},\,\iota_{\infty,K})
\]
where $H_{\dR,\infty,K}$ (resp. $H_{B,\infty}$) is a finite-dimensional $K$-vector space
with the Hodge filtration $F^\bullet$
(resp. $\Q$-vector space with the monodromy filtration $W_\bullet$), and $\iota_{\infty,K}$ is the comparison isomorphism 
\begin{equation}\label{eq:intro1}
\C\ot_K H_{\dR,\infty.K}\os{\cong}{\lra} \C\ot_\Q H_{B,\infty}, 
\end{equation}
(see \S \ref{sect:ReviewLimitMHS} for the precise definition).
By the {\it limiting periods} of $H^{n-2}(X_f/\Delta)_K$, we mean
the entries of the matrix representation of 
the map \eqref{eq:intro1}
with respect to a $K$-basis of 
$H_{\dR,\infty,K}$ and a $\Q$-basis of $H_{B,\infty}$.
By focusing on the primitive part of the de Rham cohomology group, we compute the limiting periods from the asymptotic expansions of period integrals of the following form:
\begin{equation}\label{eq:period integral intro}
    2\pi\ii\int_{\delta}{\rm Res}\left(\frac{x^{\bc-\boldsymbol{1}}}{F_f^r}\Omega\right)=\sum_{i\geq 0}h_i(t)(\log t)^i,\quad h_i(t)\in\bigcup_{N\geq 1}\C\{t^{1/N}\}[t^{-1}],
\end{equation}
where $\Omega$ is given by the formula \eqref{eq:Omega}, ${\rm Res}$ denotes Leray's residue along the vanishing locus of $F_f$ in $\P^{n-1}$, $r\in \mathbb{Z}_{>0}$ and $\bc=(c_1,\dots,c_n)\in\mathbb{Z}_{\geq 0}^n$ satisfy $c_1+\cdots+c_n=dr$.
More precisely, the limiting periods are computed from the coefficients of $h_i(t)$.
See \S\ref{sect:step2} for details.

Our main result concerns a number-theoretic description of the limiting periods for 
$X_f/\Delta$ under a certain genericity condition on the orders $\ord_t(f_1),\ldots,\ord_t(f_m)$.
Let
$A=\begin{bmatrix}
    \ba_1&\cdots&\ba_m
\end{bmatrix}$
be an $n\times m$-matrix, and let
$\Fan(A)$ denote the secondary fan (\cite[Chapter 7, Section 1.D]{GKZbook} and \cite[Chapter 8]{sturmfels1996grobner}). 
Let $\mathrm{Sk}(\Fan(A))\subset(\R^m)^*$ be the skeleton of ${\rm Fan}(A)$
(see \eqref{eq:skeleton.def} for definition).
Let $\bw_f
\in(\R^m)^*$ be the vector such
that $\bw_f\cdot\be_i=\ord_t(f_i)$ where $\{\be_i\}_{i=1,\ldots,m}$ is the
standard basis of $\R^m$.
Let $\be'_1,\ldots,\be'_n$ be the standard basis of $\R^n$, and
we define $N_A$ to be the smallest positive integer
such that $N_A[\ba'_{1}\cdots\ba'_{n}]^{-1}$ is an integer matrix
for any vectors $\ba'_1,\ldots,\ba'_n\in\{\ba_1,\ldots,\ba_m,d\be'_1,\ldots,d\be'_n\}$
such that $\det[\ba'_{1}\cdots\ba'_{n}]\ne0$.
Our main theorem is stated as follows.

\begin{theorem}
\label{thm:intro}
Suppose
$\bw_f\not\in \mathrm{Sk}(\Fan(A))$.
 Then, the limiting periods of $H^{n-2}_\infty(X_f/\Delta)_K$
lie in the $\ol K$-subalgebra of $\C$
generated by special values
\[(2\pi\ii)^{\pm1},\quad\log \overline{K}^\times,\quad
\Gamma\left(\frac{j}{N_A}\right)^{\pm1},\quad
L(k,\chi)
\]
where $j$ and $k$ run over integers
such that $0<j<N_A$ and $2\leq k$,
and $\chi$ runs over Dirichlet 
characters of conductors dividing $N_A$. 
\end{theorem}
\noindent
See Theorem \ref{thm:main} (together with Lemma \ref{lem:main}) for a more precise statement.
Theorem \ref{thm:intro} asserts only that the limiting 
periods lie in the algebra, though one is able to compute them
explicitly using Theorem \ref{thm:arbitrary triangulation}.
The assumption of Theorem \ref{thm:intro}
is equivalent to saying that $\bw_f$ corresponds
to a regular triangulation of $A$ (the convex hull of 
$\boldsymbol{0},\ba_1,\ldots,\ba_m$).
This is satisfied if 
$\ord_t(f_1),\ldots,\ord_t(f_m)$ are generic.
Furthermore, there is a simple description of the skeleton (Proposition \ref{prop:skeleton})
allowing one to readily verify the condition.
The case where $\bw_f\in \mathrm{Sk}(\Fan(A))$ remains an interesting open question.

The proof of Theorem \ref{thm:intro}
is based on the theory of {\it GKZ-systems}, different from tropical geometric approaches.
Gel'fand, Kapranov, and Zelevinsky introduced
a holonomic $D$-module arising from a pair $(A,\bc)$ consisting of a matrix 
and a complex vector, which we call the GKZ-system, denoted by $M_A(\bc)$.
A distinguishing feature is the fact that solutions to a GKZ system describe the period integrals of a generic family of hypersurfaces.
Our strategy of the proof of Theorem \ref{thm:intro} is as follows: 
firstly, it is not hard to
reduce the proof to the case where $F$ is
a {\it Fermat deformation}
\begin{equation}\label{eq:Fermat deformation intro}
\sum_{j=1}^m z_j\bx^{\ba_j}+z_{m+1}x_1^d+\cdots+z_{m+n}x_n^d.  
\end{equation}
By the assumption that $\bw_f\notin{\rm Sk}(\Fan(A))$, $\bw_f$ belongs to a maximal cone of $\Fan(A)$.
Since each maximal cone of the secondary fan corresponds to a regular triangulation of the Newton polytope of \eqref{eq:Fermat deformation intro}, we write $T$ for the corresponding regular triangulation.
$T$ defines an open subset $U_T\subset\C^{m+n}={\rm Specan}(\C[z_1,\dots,z_{m+n}])$ and we write $\Sol^T(M_A(\bc))$ for the space of (multivalued) holomorphic solutions of $M_A(\bc)$ on $U_T$.
It turns out that the period integral \eqref{eq:period integral intro} is the pull-back $f^*\varphi$ of a solution $\varphi\in\Sol^T(M_A(\bc))$.
Thus, our task is to specify the $\overline{K}$-subspace of $f^*\Sol^T(M_A(\bc))$ spanned by the period integrals \eqref{eq:period integral intro}.

To do so, we completely determine the asymptotic expansions of period integrals \eqref{eq:period integral intro} when $\bw_f$ corresponds to a canonical regular triangulation which we denote by $T(\Fer)$.
The period integrals \eqref{eq:period integral intro} are expanded into the so-called {\it Gamma series}, a multi-lateral series whose coefficients are ratios of Gamma functions.
Our next step is to apply 
the {\it connection formula} proved in \cite{matsubara2022global}
(see also Theorem \ref{thm:connection formula}).
This allows us to transform the Gamma series into 
a sum of another Gamma series 
which converge at the required degeneracy locus.
Namely, it computes the analytic continuation
\begin{equation}\label{eq:connection intro}
\Sol^{T(\Fer)}(M_A(\bc))\to\Sol^{T}(M_A(\bc))    
\end{equation}
along a certain path.
Thus, we obtain explicit asymptotic expansions of the periods \eqref{eq:period integral intro} of an arbitrary 
generic degeneration of hypersurfaces.
See Theorem \ref{thm:arbitrary triangulation} for details.
Since the limiting periods are determined by the asymptotic expansion of the period integrals, 
we can derive Theorem \ref{thm:intro}.


As we saw above, the description of the subspace of $f^*\Sol^T(M_A(\bc))$ is slightly obscure, though there is an algorithm for describing it.
To illustrate the algorithm explicitly, we introduce in \S\ref{sect:Dwork family} a class of degenerations which we name {\it perturbation of the generalized Dwork family}.
It is characterized in terms of the regular triangulation $T$ and gives rise to a class of limiting periods including those of the generalized Dwork family.
Theorem \ref{thm:analytic continuation of p.Dwork} is a closed formula for the analytic continuation \eqref{eq:connection intro}.
We note that the formula of analytic continuation employed in this paper is improved compared to the one in \cite{matsubara2022global}.

\medskip

This paper is organized as follows.
In \S\ref{sect:2}, 
we review 
the fundamental definitions and results for GKZ systems.
In particular, in \S\ref{sect:GKZ-Hypersurface},
we summarize the series solutions of GKZ systems, especially
the SST-limits (the space of limits of Gamma series) and the connection formula in detail.
This is a technical key to the proof of our main theorem.
Throughout \S\ref{sect:Fermat}, we devote ourselves to discussing the periods of the Fermat deformation.
The main result in this section is Theorem \ref{thm:arbitrary triangulation},
which allows us to compute the asymptotic expansions of period integrals explicitly.
In \S\ref{sect:limitMHS}, after a brief review of the limiting MHS, 
we prove the main theorem (Theorem \ref{thm:main}).
Finally, in \S\ref{sect:Dwork family}, we introduce the perturbation of the generalized Dwork family and describe its limiting periods explicitly.
In \S\ref{subsec:Dwork family}, we recall the limiting periods of the generalized Dwork family.
In \S\ref{subsec:p.Dwork}, we formulate the perturbation of the generalized Dwork family in terms of a regular triangulation and describe its limiting periods.

\medskip

The authors would like to thank Ryo Negishi for sharing Lemma \ref{lem:Negishi-thm}.
The first author is supported by KAKENHI Grant No. 23K0302503.
The second author is supported by KAKENHI Grant-in-Aid for Early-Career Scientists 22K13930.

\medskip

\section{GKZ systems}\label{sect:2}
\subsection{GKZ systems and secondary fans}\label{sect:GKZ system}

Let $M=\Z^N$ and $M'=\Z^n$ with positive integers $n<N$.
For an additive group $S$, we write $S_K=K\otimes_{\Z}S$ where $K=\Q$, $\R$ or $\C$.
We consider an additive homomorphism $A : M \to M'$ such that the induced morphism $M_{\Q}\to M'_{\Q}$ is a surjection.
Let
\[
A=\begin{bmatrix}\ba_1&\cdots &\ba_N\end{bmatrix}
=\begin{bmatrix}
a_{11}&\cdots &a_{1N}\\
\vdots&&\vdots\\
a_{n1}&\cdots &a_{nN}
\end{bmatrix}
\]
be the matrix representation with respect to the standard basis of $M$ and $M'$
(by abuse of notation, we use the same symbol for the matrix).
The kernel of $A$ is denoted by $L_A$.
We write $\partial_i:=\partial/\partial z_i$ for the partial derivative in $z_i$ and let
$D=\C\langle z_i,\partial_i;i=1,\dots,N\rangle$ be the Weyl algebra.
Let
$\bc=(c_1,\dots,c_n)\in M'_{\C}=\C^n$ be a complex vector.
We call a pair $(A,\bc)$
a {\it GKZ system}.
The toric ideal $I_A$ is an ideal of $\C[\partial_1,\ldots,\partial_N]$ generated by
\begin{equation}\label{eq:toric.gens}
    \prod_{i:u_i>0}\partial_i^{u_i}-\prod_{j:u_j<0}\partial_j^{u_j}\ \ (\bu=(u_1,\dots,u_N)\in L_A).
\end{equation}
The GKZ ideal $H_A(\bc)$ is the left ideal of $D$ generated by \eqref{eq:toric.gens}
and 
\begin{equation}\label{eq:GKZ.gens}
\sum_{i=1}^{m}a_{ij}z_j\partial_j+c_i\ \ (i=1,\dots,n),
\end{equation}
and put $M_A(\bc):=D/H_A(\bc)$ a left $D$-module.
Throughout this paper, we assume that $A$ is homogeneous, which means that there is a 
row vector $\bv=(v_1,\dots,v_n)\in\Q^n$ such that 
\begin{equation}\label{eq:dual vector}
    \bv \ba_j:=\sum_{i=1}^nv_ia_{ij}=1\quad (j=1,\dots,N),
\end{equation}
or equivalently $|\bu|=u_1+\cdots+u_N=0$ for all $\bu\in L_A$.
This implies that $M_A(\bc)$ is a regular holonomic $D$-module (\cite[p27, Theorem]{hotta1998equivariant}, \cite{schulze2012resonance}).

Let $\Delta_A$ denote the convex hull of $\ba_1,\dots,\ba_N$ and the origin in $M'_{\R}=\R^n$.
We write $\{\be_1,\ldots,\be_N\}$ for the standard basis of $M_\R$.
A standard simplex in $M'_{\R}$ is defined to be the set $\{\sum_{i=1}^nt_i\be_i\mid 0\leq t_i\leq 1\}$.
Let $d\mu$ be the Lebesgue measure on $M'_{\R}$ normalized so that the standard simplex 
in $M'_\R$ has volume one.
The normalized volume of $\Delta_A$ which we write by ${\rm vol}(A)$
is defined by
\begin{equation}\label{eq:normalized volume}
{\rm vol}(A)
:=
\frac{1}{[M':\Z A]}\int_{\Delta_A}d\mu,    
\end{equation}
where 
$[M':\Z A]$ denotes the cardinality of the quotient group $M'/\Z\ba_1+\cdots+\Z\ba_N$.
It is known that ${\rm vol}(A)$ agrees with the degree of the toric ideal $I_A$ 
(\cite[Theorem 4.16]{sturmfels1996grobner}).

We recall the definition of a regular polyhedral subdivision of $\Delta_A$ {(\cite[Chapter 7]{GKZbook}, \cite[Chapter 8]{sturmfels1996grobner})}. 
For any subset $\sigma$ of $\{1,\dots,N\},$ we write ${\rm cone}(\sigma)$ for the positive span of {$\{\ba_i\}_{i\in\sigma}$,} i.e., ${\rm cone}(\sigma)=\sum_{i\in\sigma}\R_{\geq 0}\ba_i.$ 
A collection {$S$} of subsets of $\{1,\dots,N\}$ is called a {\it polyhedral subdivision} if $\{{\rm cone}(\sigma)\mid \sigma\in S\}$ is the set of cones in a {polyhedral} fan whose support equals ${\rm cone}(A)=\sum_{i=1}^N\R_{\geq 0}\ba_i$.
For a real vector space $V_\R$, let $V_{\R}^*$ denote the dual vector space.
Let $\{\be^*_1,\ldots,\be_N^*\}$ be the dual basis of $\{\be_1,\ldots,\be_N\}$. 
Any $\bw=\sum_{i=1}^Nw_i \be_i^*\in M^*_{\R}$ defines a polyhedral subdivision $S(\bw)$ in the following way : a subset $\sigma\subset\{1,\dots,N\}$ belongs to $S(\bw)$ if there exists a vector ${\boldsymbol n}\in (M_{\R}')^*$ such that 
\begin{equation}\label{eq:normal vector}
    {\boldsymbol n}\cdot\ba_i=w_i\text{ if }i\in\sigma\text{ and }{\boldsymbol n}\cdot\ba_j<w_j\text{ if } j\notin{\sigma},
\end{equation}
where we think $\ba_i$ of being a vector in $M'_\R$, and the dot product denotes
the duality pairing $(M'_\R)^*\ot M'_\R\to\R$.
A polyhedral subdivision $S$ is called a {\it regular polyhedral subdivision} if $S=S(\bw)$ for some $\bw$.
\begin{remark}\label{rem:2.1}
An element $\sigma\in S(\bw)$ is said to be independent if $\{ \ba_i\mid i\in\sigma\}$ is linearly independent.
This is equivalent to that $\{ (w_i,\ba_i)\mid i\in\sigma\}$ is linearly independent in $\R\oplus M'_{\R}$.
In fact, if $\{ \ba_i\mid i\in\sigma\}$ is linearly dependent, there exists $c_i\in\R$ $(i\in\sigma)$ such that $\sum_{i\in\sigma}c_i\ba_i=0$.
Then, \eqref{eq:normal vector} implies that $\sum_{i\in\sigma}c_iw_i=0$.
Thus, it follows that $\sum_{i\in\sigma}c_i(w_i,\ba_i)=0$.    
\end{remark}
A subset $\sigma\subset\{1,\dots,N\}$ is called a {\it simplex} if the cardinality $\sharp \sigma$ of $\sigma$ equals $n$ and it is independent.
If any maximal (with respect to inclusion) element $\sigma$ of a regular polyhedral subdivision $T$ is a simplex, we call $T$ a {\it regular triangulation}. 
For a subset $\sigma\subset\{1,\dots,N\}$, we write $C_\sigma$ for the set of $\bw\in M^*_{\R}$ which satisfies \eqref{eq:normal vector}.
Given a regular polyhedral subdivision $S$, we write $C_S$ for the cone consisting of vectors $\bw$ such that $S(\bw)=S$.
Then, by definition, we obtain a relation
\begin{equation}
    C_S=\bigcap_{\sigma\in S}C_\sigma\subset M_{\R}^*.
\end{equation}
The set of cones ${\rm Fan}(A):=\{ C_S\mid S \text{ is a regular polyhedral subdivision}\}$ forms a polyhedral fan.
We call ${\rm Fan}(A)$ the {\it secondary fan} of $A$.
It is readily seen from the definition that each $C_\sigma$ contains 
the linear subspace $(M'_{\R})^*\subset M^*_\R$, where 
the inclusion is given by the transpose of $A$.
Therefore, the secondary fan is identified with a polyhedral fan in $(L_{A,\R})^*\simeq M_{\R}^*/(M_{\R}')^*$.
Let $\pi_A:M_{\R}^*\rightarrow (L_{A,\R})^*$ be the linear map induced from the inclusion
$L_A\hookrightarrow M$.
Put ${\boldsymbol g}_i:=\pi_A(\be_i^*)\in (L_{A,\R})^*$.
The set $\{{\boldsymbol g}_1,\dots,{\boldsymbol g}_N\}$ is called
the {\it Gale dual} of $A$.
Each cone $\pi_A(C_S)$ is described in terms of Gale dual as follows.
For a subset $\sigma\subset\{ 1,\dots,N\}$ and a subset $S\subset 2^{\{1,\dots,N\}}$, we set 
\begin{equation}\label{eq:bar C}
\bar{C}_\sigma:=\sum_{i\notin\sigma}\R_{>0}{\boldsymbol g}_i\subset (L_{A,\R})^*.    
\end{equation}
Then, for a regular subdivision $S$, the equality
\begin{equation}\label{eq:bar C sigma gale}
\pi_A(C_S)=\bigcap_{\sigma\in S}\bar{C}_\sigma
\end{equation}
holds as in \cite[p20, Corollary 1]{stienstra2007gkz}.
We put
\begin{equation}\label{eq:skeleton.def}
    \mathrm{Sk}(\Fan(A)):=\bigcup_{\dim C_S<N} C_S\subset M_\R^*
\end{equation}
and call it the {\it skeleton} of $\Fan(A)$.
The explicit description of the skeleton of ${\rm Fan}(A)$ is given in the following proposition.

\begin{proposition}\label{prop:skeleton}
The following identity holds true.
    \begin{equation}\label{eq:skeleton}
    \pi_A({\rm Sk}({\rm Fan}(A)))=\bigcup_{1\leq i_1<\dots<i_{N-n-1}\leq N}\left(\R_{\geq 0}{\boldsymbol g}_{i_1}+\cdots+\R_{\geq 0}{\boldsymbol g}_{i_{N-n-1}}\right).
    \end{equation}
Furthermore, ${\rm Sk}({\rm Fan}(A))$ is the preimage of \eqref{eq:skeleton} by $\pi_A$.
\end{proposition}

\begin{proof}
    It is obvious by the construction that the left-hand side of \eqref{eq:skeleton} is included in the right-hand side.
    Let us prove the reverse inclusion.
    We fix an element $\pi_A(\bw)\in \R_{\geq 0}{\boldsymbol g}_{i_1}+\cdots+\R_{\geq 0}{\boldsymbol g}_{i_{N-n-1}}$.
    It is enough to show that $S(\bw)$ is not a triangulation.
    Without loss of generality, we may assume that $i_1=1,\dots,i_{N-n-1}=N-n-1$ and there is a number $1\leq r\leq N-n-1$ so that ${\rm span}_{\R}\{{\boldsymbol g}_{i_1},\dots,{\boldsymbol g}_{i_{N-n-1}}\}={\rm span}_{\R}\{{\boldsymbol g}_1,\dots,{\boldsymbol g}_r\}$.
    Then, by \cite[p20, Corollary 1]{stienstra2007gkz}, it follows that $\sigma=\{r+1,\dots,N\}$ appears in $S(\bw)$.
    The cardinality of $\sigma$ is strictly larger than $n$, hence it is not a simplex.
    Therefore, $S(\bw)$ is not a triangulation.
    This completes the proof of \eqref{eq:skeleton}.
    The last statement follows from the fact that every $C_\sigma$ include the linear subspace $(M'_\R)^*$.
\end{proof}

\medskip

The following lemma will be used in the proof of
Theorem \ref{thm:main}.

\begin{lemma}\label{lem:compare.skeleton}
Let $A=\begin{bmatrix}\ba_1&\cdots&\ba_N\end{bmatrix}$ be as before, and let
$A'=\begin{bmatrix}\ba_1&\cdots&\ba_N&d\be_1&\cdots &d\be_n
    \end{bmatrix}$ 
    where $\be_i$ is the $i$-th standard vector.
    For any $\bw=(w_1,\dots,w_N)\in\R^N$ such that the corresponding regular subdivision $S(\bw)$ of $A$ is a regular triangulation, there exists $w_{N+1},\cdots,w_{n+N}\in\R$ such that the the regular subdivision $S(\bw')$ of $A'$ corresponding to $\bw'=(w_1,\dots,w_{n+N})\in\R^{n+N}$ is a regular triangulation.
    Moreover, $w_{N+1},\cdots,w_{n+N}\in\R$ can be chosen to be arbitrarily large enough.
\end{lemma}

\begin{proof}
    Given a circuit ($=$ minimal linearly dependent set) $\sigma\subset\{1,\dots,n+N\}$, we choose $c_i(\sigma)\in\R\setminus\{ 0\}$ $(i\in\sigma)$ so that $\sum_{i\in\sigma}c_i(\sigma)\ba_i=0$.
    Note that the vector $(c_i(\sigma);i\in\sigma)$ is unique up to a scalar multiplication.
    We take $w_{N+1},\cdots,w_{n+N}\in\R$
    so that
    \begin{equation}\label{eq:choice of omega prime}
        \sum_{i\in\sigma}c_i(\sigma)w_i\neq 0
    \end{equation}
    for any circuit $\sigma\subset\{1,\dots,n+N\}$ such that $\sigma\not\subset\{1,\dots,N\}$.
    Assume that $S(\bw')$ is not a regular triangulation.
    Let $\sigma\in S(\bw')$ be a circuit, which exists by Remark \ref{rem:2.1}.
    Let $\ba'_i$ denote the $i$-th vector of $A'$ $(i=1,\dots,n+N)$.
    Then, there exists a vector ${\boldsymbol n}\in (\R^n)^*$ such that
    \begin{equation}\label{eq:regular subdivision}
        {\boldsymbol n}\cdot \ba'_i=w_i\quad\text{if}\ \  i\in\sigma\quad\text{and}\quad {\boldsymbol n}\cdot \ba'_i<w_i\quad\text{if}\ \  i\notin\sigma.
    \end{equation}
    If $\sigma\subset\{1,\dots,N\}$, it follows from \eqref{eq:regular subdivision} that $\sigma$ appears in $S(\bw)$.
    This contradicts the assumption that $S(\bw)$ is a regular triangulation.
    Therefore, we may assume that $\sigma\not\subset\{1,\dots,N\}$.
    Then, it follows that $\sum_{i\in\sigma}c_i(\sigma)w_i=0$, which clearly contradicts \eqref{eq:choice of omega prime}.
    \end{proof}

\subsection{Series solutions of GKZ systems}\label{sect:GKZ system sol}

For a vector $\bg=(\gamma_1,\ldots,\gamma_N)\in M_{\C}=\C^N$ such that $A\bg=-\bc$, we call a formal series 
\begin{equation}\label{eq:def of Gamma series}
    \varphi(\bg;\bz):=\sum_{\bu\in L_A}
    \frac{\bz^{\bu+\bg}}{\Gamma(1+\bu+\bg)}
\end{equation}
a {\it Gamma series},
where we use the following notation
\[
\bz^{\bu+\bg}:=\prod_{i=1}^Nz_i^{u_i+\gamma_i},
\quad \Gamma(1+\bu+\bg):=\prod_{i=1}^N\Gamma(1+u_i+\gamma_i).
\]
Consider a regular triangulation $T$ of $A$.
Let $T_{\max}$ denote the set of maximal elements of $T$ (hence any $\sigma\in T_{\max}$ is a simplex).
For each $\sigma\in T_{\max}$ and $\bp=(p_i)_{i\not\in \sigma}$ $(p_i\in\Z)$, 
there is a unique $\bg\in \C^N$ such that  
\begin{equation}\label{eq:char. exponent}
A\bg=-\bc\quad\text{and}\quad \gamma_i=p_i\quad (i\notin \sigma).    
\end{equation}
We write it by $\bg^{\bc}_{\sigma,\bp}$, which is often referred to as an exponent for
the GKZ system $(A,\bc)$. 
Put a set
\begin{equation}\label{eq:series solutions}
    \Phi_T:=\{\varphi(\bg^{\bc}_{\sigma,\bp};\bz)\mid \sigma\in T_{\max},\ \bp=(p_i)_{i\not\in \sigma}\in\Z^{N-n}\}.
\end{equation} 
This is a finite set, and any element is a (multi-valued) function that converges in a non-trivial open subset $U_T\subset \C^N$.
To be more precise, let ${\rm Log}:(\C^\times)^N\to \R^N$ be the map defined by 
\begin{equation}\label{eq:Log}
    {\rm Log}(\bz)=(-\log|z_1|,\dots,-\log|z_N|).
\end{equation}
Then, it follows that there exists an element $v_T\in C_T=\cap_{\sigma\in T}C_\sigma$ such that the following inclusion holds true:
\begin{equation}\label{eq:U_T is non-empty}
 {\rm Log}^{-1}(v_T+C_T)\subset U_T.
\end{equation}
Let $\O^{\text{an}}_{\bz}$ denote the germ of holomorphic functions at $z\in\C^N$.
For a left $D$-module ${\mathcal M}$,
we define the solution space at $T$ by $\Sol^{T}({\mathcal M})=\Hom_D({\mathcal M},\O^{\text{an}}_{\bz})$
for a point $\bz\in\C^N$ such that $\mathrm{Log}(\bz)=r\,\bw_T$ with some interior point $\bw_T$ of $C_T$ and $r\gg0$.
If ${\mathcal M}$ is an integrable connection in a neighborhood of the real torus 
${\rm Log}^{-1}(r\bw_T)$, then the fundamental group $\pi_1({\rm Log}^{-1}(r\bw_T),\bz)$ acts on 
$\Sol^{T}({\mathcal M})$
in the natural way,
\begin{equation}\label{eq:local monodromy}
\pi_1({\rm Log}^{-1}(r\bw_T),\bz)\lra \mathrm{GL}(\Sol^{T}({\mathcal M})),
\end{equation}
which is referred to as the {\it local monodromy} on the solution space.

\medskip

Following \cite[Section 3]{fernandez2010irregular}, we say that a parameter $\bc\in\C^n$ is {\it very generic} if for any $\sigma\in T_{\max}$ and $\bp=(p_i)_{i\not\in\sigma}$, the $j$-th element of $\bg^{\bc}_{\sigma,\bp}$ is not an integer for all $j\in\sigma$.
As a choice of vectors $\{\bp_1(\sigma),\dots,\bp_{r(\sigma)}(\sigma)\}$ so that it gives rise to a finite presentation $\Phi_T=\{\varphi(\bg^{\bc}_{\sigma,\bp};\bz);\sigma\in T_{\max},\bp=\bp_1(\sigma),\dots,\bp_{r(\sigma)}(\sigma)\}$ of \eqref{eq:series solutions},
it is convenient to choose the top-dimensional standard pairs in the sense of \cite[\S 3.2]{saito2013grobner} where $r(\sigma)$ is ${\rm vol}(\sigma)$ in view of \cite[Theorem 8.8]{sturmfels1996grobner}.
In view of the local monodromy around the coordinates, it is readily seen that the set of series \eqref{eq:series solutions} is linearly independent over $\C$ and its cardinality is $\vol(A)$.
Moreover, $\bc\in\C^n$ is semi-nonresonant in the sense of \cite[Section 5]{adolphson1994hypergeometric} if it is very generic. 
Thus, \cite[Theorem 5.15]{adolphson1994hypergeometric} proves the following assertion.

\begin{theorem}
\label{thm:eq:condition *}
Suppose that the parameter $\bc\in\C^n$ is very generic.
Then, $\dim\Sol^T(M_A(\bc))=\vol(A)$, and 
the set $\Phi_T$ in \eqref{eq:series solutions} forms a $\C$-basis of $\Sol^{T}(M_A(\bc))$.
\end{theorem}


Suppose that $\bc=(c_1,\ldots,c_n)$ is a real vector, so that any $\bg^{\bc}_{\sigma,\bp}$ is a real vector.
Fix a vector $\bw\in C_T\subset M_\R^*$.
For a vector $\ba\in M_\R=\R^N$ and a monomial $\bz^{\ba}=\prod_{i=1}^Nz_i^{a_i}$, 
we define its $\bw$-weight by $\bw\cdot\ba$ where 
the dot product denotes the linear map $M^*_\R\ot M_\R\to\R$.
For $k\in\R$, let $\mathbf{in}_{\leq k}(\varphi(\bg^{\bc}_{\sigma,\bp};\bz))$ denote
the sum of the terms corresponding to monomials with $\bw$-weights less than or equal to $k$.

\begin{lemma}\label{lem:finite sum}
For any $\sigma\in T_{\max}$ and $\bp=(p_i)_{i\notin\sigma}\in\Z^{N-n}$, 
$\mathbf{in}_{\leq k}(\varphi(\bg^{\bc}_{\sigma,\bp};\bz))$ is a finite sum. 
\end{lemma}

\begin{proof}
Write $\bg=\bg^{\bc}_{\sigma,\bp}$ for simplicity.
By the definition \eqref{eq:def of Gamma series}, the indices $\bu$ in $\varphi(\bg;\bz)$
run over the elements of $L_A$ such that $u_j+p_j\geq 0$ for all $j\not\in\sigma$.
Since the $\bw$-weight of $\bz^{\bg+\bu}$ is $\bw\cdot(\bg+\bu)$,  the indices $\bu$ 
in $\mathbf{in}_{\leq k}(\varphi(\bg;\bz))$ runs over the set
\begin{equation}\label{eq:finite set}
\{\bu\in L_A\mid \pi_A(\bw)\cdot\bu\leq k-\bw\cdot\bg\text{ and } \boldsymbol{g}_j\cdot\bu\geq -p_j\,(j\not\in\sigma)\}
\end{equation}
where $\{\boldsymbol{g}_j\}$ are the Gale dual.
By \eqref{eq:bar C sigma gale}, $\pi_A(\bw)$ belongs to
the $(N-n)$-dimensional open cone $\bar C_\sigma=\sum_{j\not\in \sigma} \R_{>0}\boldsymbol{g}_j$.
Therefore the finiteness of the set \eqref{eq:finite set} follows from the fact that
 for any $R_1,R_2\in\R$,
\[
\{\bu\in L_{A,\R}\mid \pi_A(\bw)\cdot\bu\leq R_1\text{ and } \boldsymbol{g}\cdot\bu\geq R_2
\text{ for all }\boldsymbol{g}\in\bar C_\sigma\}
\]
is a compact set.

\end{proof}


When the condition in Theorem \ref{thm:eq:condition *} is violated, the set of solutions \eqref{eq:series solutions} may only span a subspace of  $\Sol^{T}(M_A(\bc))$ .
However, the space of solutions can be recovered by taking a suitable limit, which we call a {\it SST-limit}.

\begin{definition}\label{defn:SST-limit defn}
Let $\bc \in\Q^n$ be an arbitrary vector.
Let $\bc'\in \Q^n$ be a generic vector so that $\bc(\varepsilon):=\bc+\varepsilon\bc'$ fulfills the condition in Theorem \ref{thm:eq:condition *} for any $0<|\varepsilon|\ll 1$.
We define the SST-limit 
\begin{equation}\label{eq:SST-def}
    \lim_{\varepsilon\to0}\Sol^T(M_A(\bc(\varepsilon)))_\C\subseteq\Sol^{T}(M_A(\bc))
\end{equation}
to be the subspace of all the limits
$\lim_{\varepsilon\to 0}\sum_{i} C_i(\varepsilon)\varphi_i(\bz)$,
where the sum is a finite sum and $\varphi_i(\bz)\in\Sol^T(M_A(\bc(\varepsilon)))$ and $C_i(\varepsilon)
\in\C(\!(\varepsilon)\!)=\C[\![\varepsilon]\!][\varepsilon^{-1}]$ are chosen so that the limit exists.
\end{definition}
\begin{remark}\label{remark:SST-limit defn}
The local monodromy \eqref{eq:local monodromy} on $\Sol^T(M_A(\bc(\varepsilon)))$ 
factors through a finite quotient by Theorem \ref{thm:eq:condition *}.
Let 
\[
\Sol^T(M_A(\bc(\varepsilon)))=\bigoplus_\chi\Sol^T(M_A(\bc(\varepsilon)))(\chi)
\]
be the simultaneous eigenspace decomposition of the local monodromy \eqref{eq:local monodromy} by the characters $\chi$.
Let $\{\varphi^\chi_i(\bz)\}_i$ be a $\C$-basis of
the $\chi$-component. Then a limit
$\lim_{\varepsilon\to 0}\sum_\chi\sum_{i} C_i(\varepsilon)\varphi^\chi_i(\bz)$ exists
if and only if so does 
$\lim_{\varepsilon\to 0}\sum_{i} C_i(\varepsilon)\varphi^\chi_i(\bz)$ 
for each $\chi$.
\end{remark}
The following theorem is a consequence of \cite[Theorem 5.15]{adolphson1994hypergeometric} 
and the proof of \cite[Theorem 3.5.1]{saito2013grobner}.
\begin{theorem}\label{thm:SST limits span the solution space}
    If $\bc$ is in the interior of ${\rm pos}(A):=\sum_i\R_{\geq0}\ba_i$, then 
the holonomic rank of $M_A(\bc)$ is $\vol(A)$, and 
    the equality holds in \eqref{eq:SST-def}.
    In particular, the SST-limit does not depend on the choice of $\bc'$.
\end{theorem}
Any element of the SST-limit is of the following form
\begin{equation}\label{eq:log series}
    \varphi(\bz)=\sum_{i_1,\ldots,i_N\geq0}g_{i_1\ldots i_N}(\bz)(\log z_1)^{i_1}\cdots(\log z_N)^{i_N}
\end{equation}
where the sum is a finite sum and $g_{i_1\ldots i_N}(\bz)$ are multi-lateral series of the form $\sum_{\bu\in L_A}c_{\bu}z^{\bu+\bg}$ for some $\bg\in\C^N$.
We set $\mathbf{in}_{\leq k}(\varphi(\bz)):=\sum\mathbf{in}_{\leq k}(g_{i_1\ldots i_N}(\bz))(\log z_1)^{i_1}\cdots(\log z_N)^{i_N}$.
The following statement is an immediate consequence of Lemma \ref{lem:finite sum}.
\begin{lemma}
    \label{prop:finite truncation}
    Suppose that $\bc$ is in the interior of ${\rm pos}(A)$.
    Then, $\mathbf{in}_{\leq k}(\varphi(\bz))$ is a finite sum
for any $\varphi(\bz)\in \lim_{\varepsilon\to0}\Sol^T(M_A(\bc(\varepsilon)))_\C$ and $k\in\R$.
\end{lemma}
\begin{definition}
Let $\bc(\varepsilon)=\bc+\varepsilon\bc'$ be as in Definition \ref{defn:SST-limit defn}.
Let $\{\bg^{\bc(\varepsilon)}_i\}$ be exponents so that
$\Phi_T=\{\varphi(\bg^{\bc(\varepsilon)}_i;\bz);1\leq i\leq \vol(A)\}$.
For a $\Q$-subalgebra $R\subset \C$, we define the SST-limit
\begin{equation}\label{eq:SST-limit over R}
\lim_{\varepsilon\to0}\Sol^T(M_A(\bc(\varepsilon)))_R\subset \Sol^T(M_A(\bc))
\end{equation} 
over $R$
to be the $R$-linear span of the limits
$\lim_{\varepsilon\to0}\sum_iC_i(\varepsilon)\varphi(\bg_i^{\bc(\varepsilon)};\bz)$
for
$C_i(\varepsilon)\in R(\!(\varepsilon)\!)$.
\end{definition}
\begin{remark}
The authors do not know whether the SST-limit over $R$ is independent of the choice of
$\bc'$.
\end{remark}

\begin{definition}\label{defn:psi-algebra}
    For a positive integer $M\geq1$ and a rational number $e\in \Q$, we define 
    $\PG(M,e)$ to be the $\Q$-subalgebra of $\C$
    generated by special values
\[
\prod_{j=1}^r\Gamma\left(q_j\right)^{m_j},\quad
\psi^{(0)}(p_{0})-\psi^{(0)}(1),\quad
\psi^{(k)}(p_{k})
\]   
where $q_j,p_{k}\in\frac{1}{M}\Z_{\geq1}$, $r\in\Z_{\geq0}$,
$k\in\Z_{\geq1}$, $m_j\in\Z$ and they are subject to constraints 
$\sum_{j=1}^rm_jq_j\equiv e\mod \Z$.
Here, $\psi^{(i)}(z):=\dfrac{d^i}{dz^i}\dfrac{\Gamma'(z)}{\Gamma(z)}$
are the polygamma functions, cf. \cite[Chap.5]{olver2010nist}.
\end{definition}

\begin{theorem}\label{thm:coefficients}
Suppose that $\bc\in\Q^n$ be in the interior of ${\rm pos}(A)$.
Let $\bc(\varepsilon)=\bc+\varepsilon\bc'$ be as in Definition \ref{defn:SST-limit defn}.
Let $\{\bg_i^{\bc(\varepsilon)}\}$ be exponents so that
$\Phi_T=\{\varphi(\bg^{\bc(\varepsilon)}_i;\bz)\}_i$.
Let $N_{A,\bc,T}$ be the smallest positive integer
such that $\bg_i^{\bc}\in(N_{A,\bc,T})^{-1} \Z^N$ 
for all $i$.
Let 
$R:=\PG(N_{A,\bc,T},-\bv\cdot\bc)$
be the $\Q$-algebra defined in Definition \ref{defn:psi-algebra}
where $\bv$ is as in \eqref{eq:dual vector}.
Then, the SST-limit 
\begin{equation}\label{eq:SST-limit}
\lim_{\varepsilon\to0}\Sol^T(M_A(\bc(\varepsilon)))_R
\end{equation}
over $R$ includes
 a $\C$-basis of $\Sol^T(M_A(\bc))$.
Furthermore, for any element
 \[
\sum_{i_1,\ldots ,i_N}g_{i_1\ldots i_N}(\bz)(\log z_1)^{i_1}\cdots(\log z_N)^{i_N} \in \lim_{\varepsilon\to0}\Sol^T(M_A(\bc(\varepsilon)))_R,
 \]
all coefficients 
of $g_{i_1\ldots i_N}(\bz)$ lie in the ring $R$.
The same statement is true if we replace $R$ with $R_{K}:=\Image(K\ot R\to\C)$
for a $\Q$-subalgebra $K\subset\C$.
\end{theorem}
\begin{proof}
We begin with a formula
\[
\Gamma(a+\varepsilon)/\Gamma(a)=
\exp\left(\psi(a)\varepsilon+\frac12\psi^{(1)}(a)\varepsilon^2+\frac16\psi^{(2)}(a)\varepsilon^3+\cdots\right)
\]
for $a\not\in \Z_{\geq0}$ and $0<|\varepsilon|<1$.
Let $\bg^{\bc(\varepsilon)}=\bg+\varepsilon\boldsymbol{\delta}$ be an exponent
for the GKZ-system $M_A(\bc(\varepsilon))$.
We first describe the series
\[
\varphi({\bg^{\bc(\varepsilon)}};\bz)=
\sum_{\bu\in L_A}
\frac{\bz^{\bu+\bg+\varepsilon\boldsymbol{\delta}}}
{\Gamma(1+\bu+\bg+\varepsilon\boldsymbol{\delta})}
\]
by an expansion of $\varepsilon$.
If $\gamma_i\not\in \Z_{<0}$, then
\begin{align*}
\frac{1}
{\Gamma(1+u_i+\gamma_i+\varepsilon\delta_i)}
=&
\frac{1}{\Gamma(1+u_i+\gamma_i)}
\exp\left(-\sum_{n=1}^\infty\frac{\psi^{(n-1)}(1+u_i+\gamma_i)}{n!}(\varepsilon\delta_i)^n\right)
 \\
=
&\frac{1}{\Gamma(1+\gamma_i)}
\exp\left(-\sum_{n=1}^\infty\frac{\psi^{(n-1)}(1+\gamma_i)}{n!}(\varepsilon\delta_i)^n\right)\times
 \\
&
\frac{\Gamma(1+\gamma_i)}{\Gamma(1+u_i+\gamma_i)}
\exp\left(\sum_{n=1}^\infty\frac{\psi^{(n-1)}(1+\gamma_i)-\psi^{(n-1)}(1+u_i+\gamma_i)}{n!}(\varepsilon\delta_i)^n\right).
\end{align*}
We note that the last term belongs to the ring $\Q[\![\varepsilon]\!]$,
since $\Gamma(\gamma+k)/\Gamma(\gamma)\in\Q^\times$
and $\psi^{(i)}(\gamma+k)-\psi^{(i)}(\gamma)\in \Q$ 
for any $k\in\Z$ and $\gamma\in\Q\setminus\Z$.
If $u_i+\gamma_i\in \Z_{<0}$, then
\begin{equation}\label{eq:1/Gamma 1}
\begin{split}
\frac{1}
{\Gamma(1+u_i+\gamma_i+\varepsilon\delta_i)}
=&\frac{\varepsilon\delta_i(\varepsilon\delta_i-1)\cdots(\varepsilon\delta_i+1+u_i+\gamma_i)}{\Gamma(1+\varepsilon\delta_i)}\\
=&\varepsilon\delta_i(\varepsilon\delta_i-1)\cdots(\varepsilon\delta_i+1+u_i+\gamma_i)
\exp\left(-\sum_{n=1}^\infty\frac{\psi^{(n-1)}(1)}{n!}(\varepsilon\delta_i)^n\right)
\end{split}
\end{equation}
and, 
if $u_i+\gamma_i\in \Z_{\geq0}$, then
\begin{equation}\label{eq:1/Gamma 2}
\begin{split}
\frac{1}
{\Gamma(1+u_i+\gamma_i+\varepsilon\delta_i)}
=&\frac{1}{(\varepsilon\delta_i+1)\cdots(\varepsilon\delta_i+u_i+\gamma_i)\Gamma(1+\varepsilon\delta_i)}\\
=&\frac{1}{(\varepsilon\delta_i+1)\cdots(\varepsilon\delta_i+u_i+\gamma_i)}
\exp\left(-\sum_{n=1}^\infty\frac{\psi^{(n-1)}(1)}{n!}(\varepsilon\delta_i)^n\right).
\end{split}
\end{equation}
To sum up the above results, we introduce the following notation:
\[
\Gamma_\Z(z):=\begin{cases}
\Gamma(z)&(z\not\in\Z)\\
1&(z\in\Z),
\end{cases}
\quad
\psi^{(i)}_{\Z}(z):=\begin{cases}
\psi^{(i)}(z)&(z\not\in\Z)\\
\psi^{(i)}(1)&(z\in\Z)
\end{cases}
\]
and
\[
\widetilde\psi^{(i)}_\Z(z):=\begin{cases}
\psi^{(0)}_\Z(z)-\psi^{(0)}_\Z(1)&(i=0)\\
\psi^{(i)}_\Z(z)&(i\geq1).
\end{cases}
\]
The formulas \eqref{eq:1/Gamma 1} and \eqref{eq:1/Gamma 2} imply that there are formal series $\phi_{\bu,\bg_{\bp},\boldsymbol{\delta}_{\bp}}(\varepsilon)\in \Q[\![\varepsilon]\!]$ such that an equality
\begin{equation}
\varphi({\bg^{\bc(\varepsilon)}};\bz)
= \frac{\exp(-\psi^{(0)}(1)\varepsilon|\boldsymbol{\delta}|)}{\Gamma_\Z(1+\bg)}
\exp\left(-\sum_{n=1}^\infty \sum_{i=1}^N\frac{\widetilde\psi^{(n-1)}_\Z(1+\gamma_i)}{n!}(\varepsilon\delta_i)^n\right) 
\sum_{\bu\in L_A}\phi_{\bu,\bg,\boldsymbol{\delta}}(\varepsilon)
\bz^{\bu+\bg+\varepsilon\boldsymbol{\delta}}
\label{eq:lemma 2.5.eq1}
\end{equation}
holds true.

\medskip

We turn to the proof of the theorem.
Let $\bg_k^{\bc(\varepsilon)}=\bg_k+\varepsilon\boldsymbol{\delta}_k$,
and $\bg_k=(\gamma_{k,i})$ and $\boldsymbol{\delta}_k=(\delta_{k,i})$.
By Remark \ref{remark:SST-limit defn}, it is enough to consider the limits among the series $\varphi({\bg^{\bc(\varepsilon)}_1};\bz),\ldots,
\varphi({\bg^{\bc(\varepsilon)}_s};\bz)$ for $\bg_1\equiv\cdots\equiv\bg_{s}\mod \Z^N$.
We note that $|\boldsymbol{\delta_k}|$
does not depend on $k$ by the assumption \eqref{eq:dual vector} together with the fact $A\boldsymbol{\delta}_k=\bc'$.
We put
$e(\varepsilon):=\exp(-\psi^{(0)}(1)\varepsilon|\boldsymbol{\delta}_k|)$.
Noticing that $\widetilde\psi^{(n-1)}_\Z(\gamma+k)-\widetilde\psi^{(n-1)}_\Z(\gamma)\in\Q$ and
$\Gamma_\Z(\gamma+k)/\Gamma_\Z(\gamma)\in\Q^\times$
for any $\gamma\in\Q$ 
and $k\in\Z$,
one can rewrite \eqref{eq:lemma 2.5.eq1} for $\bg^{\bc(\varepsilon)}=\bg^{\bc(\varepsilon)}_k$ as 
\begin{equation}
\varphi({\bg^{\bc(\varepsilon)}_k};\bz)
= \frac{e(\varepsilon)}{\Gamma_\Z(1+\bg_1)}
\exp\left(-\sum_{n=1}^\infty \sum_{i=1}^N\frac{\widetilde\psi_\Z^{(n-1)}(1+\gamma_{1,i})}{n!}(\varepsilon\delta_{k,i})^n\right) 
\sum_{\bu\in L_A}\widetilde\phi_{\bu,\bg_k,\boldsymbol{\delta}_k}(\varepsilon)
\bz^{\bu+\bg_k+\varepsilon\boldsymbol{\delta}_k},
\label{eq:lemma 2.5.eq2}
\end{equation}
where $\widetilde\phi_{\bu,\bg_k,\boldsymbol{\delta}_k}(\varepsilon)\in\Q[\![\varepsilon]\!]$ .
Therefore, to obtain the $\vol(A)$-dimensional space of solutions, 
it is enough to take the SST-limits 
\begin{equation}
\label{eq:lemma 2.5.eq3}
\lim_{\varepsilon\to0}\sum_k C_k(\varepsilon)
\varphi({\bg^{\bc(\varepsilon)}_k};\bz)=\sum_{k,i}\sum_{\bu}c_{\bu,k,i}\bz^{\bu+\bg_k}
(\log \bz^{\boldsymbol{\delta}_k})^i
\end{equation}
where 
$C_k(\varepsilon)$ are  
Laurent polynomials whose coefficients lie in a ring
\begin{equation}
\label{eq:lemma 2.5.eq4}
R':=\Q\left[\widetilde\psi^{(n-1)}_\Z(1+\gamma_{1,i})
\right]_{1\leq n\leq m,\, 1\leq i\leq N}
\end{equation}
with some $m\gg1$.
This is a subring of
$R=\PG(N_{A,\bc,T},-\bv\cdot\bc)$.
This shows that the SST-limit $\lim_{\varepsilon\to0}\Sol^T(M_A(\bc(\varepsilon)))_R$ over $R$ contains a $\C$-basis of
$\lim_{\varepsilon\to0}\Sol^T(M_A(\bc(\varepsilon)))_\C$.
Moreover, we may delete $e(\varepsilon)$ from \eqref{eq:lemma 2.5.eq2}
to compute \eqref{eq:lemma 2.5.eq3}
as $e(0)=1$,
and then one sees that all the coefficients $c_{\bu,k,i}$ in 
\eqref{eq:lemma 2.5.eq3} lie in $\Gamma_\Z(1+\bg_1)^{-1}R'\subset R$.
This completes the proof for $R$.
The last statement for $R_K$ is obvious from the above discussion.
\end{proof}

The following theorem summarizes \cite[Theorem 3.4.1 and Equation 2.11]{matsubara2022global} and plays a key role in this paper.
\begin{theorem}[Connection formula]\label{thm:connection formula}
Suppose that $\bc\in\C^n$ is very generic.
Let $T_1, T_2$ be regular triangulations of $A$ which share a facet.
Then, there is a path $\gamma$ from a point $\bz_1\in U_{T_1}$ to
a point $\bz_2\in U_{T_2}$ with the following property: let
\[
\rho_\gamma:\Sol^{T_1}(M_A(\bc))\os{\cong}{\lra}\Sol^{T_2}(M_A(\bc))
\]
be an isomorphism induced by the analytic continuation along $\gamma$.
The matrix representation 
of $\rho_\gamma$ with respect to the basis $\Phi_{T_1}$ and $\Phi_{T_2}$ takes values in the ring 
$$
{\Q}\left[e^{\pi\ii/M},e^{\pm\pi\ii(\gamma_{\sigma,\bp}^{\bc})_i},\left(\frac{\pi}{\sin\pi(\gamma_{\sigma,\bp}^{\bc})_i}\right)^{\pm 1};\sigma\in T_{1,\max}\cup T_{2,\max},\bp=(p_i;i\notin\sigma),i\in\sigma\right],
$$
where $M$ is the maximum of volumes of simplices $\sigma\in T_{1,\max}\cup T_{2,\max}$.
Moreover, the entries of the matrix are explicitly computable. 
\end{theorem}
The following lemma is an immediate consequence of Theorem \ref{thm:connection formula},
which we shall use in the proof of the main theorem (Theorem \ref{thm:main}).
\begin{lemma}\label{prop:connection formula}
Suppose that $\bc\in\Q^n$ is the interior of ${\rm {pos}}(A)$.
Let $\bc(\varepsilon)=\bc+\varepsilon\bc'$ be as in Definition \ref{defn:SST-limit defn}.
Let $T_1, T_2$ be regular triangulations of $A$ that share a facet.
Let $N_{T_1,T_2}$ be the least common multiple of $N_{A,\bc,T_1}$ and $N_{A,\bc,T_2}$, and put 
$R:=\PG(N_{T_1,T_2},-\bv\cdot\bc)$ (see Definition \ref{defn:psi-algebra}). Let $\widetilde R\subset\C$ be the $\Q$-algebra generated by
$R$ and the ring in Theorem \ref{thm:connection formula}.
Then 
\[
\rho_\gamma\bigg(\lim_{\varepsilon\to0}\Sol^{T_1}(M_A(\bc(\varepsilon)))_{\widetilde R}\bigg)=
\lim_{\varepsilon\to0}\Sol^{T_2}(M_A(\bc(\varepsilon)))_{\widetilde R}.
\]
\end{lemma}

\subsection{GKZ-systems and Periods of Projective Hypersurfaces}
\label{sect:GKZ-Hypersurface}
Let
$\bx=(x_1,\ldots,x_n)$ be the homogeneous coordinates of $\P^{n-1}_K={\rm Proj}(K[x_1,\dots,x_n])$ over a field $K$ of
characteristic zero.
Let $m\geq n$ and $d>0$ be integers, let $z_i$ be an indeterminate for $1\leq i\leq m$ and let
\[
F=F(\bx;\bz)=\sum_{i=1}^mz_i \bx^{\ba_i}
\]
be a homogeneous polynomial of degree $d$ such that the hypersurface $F=0$ is 
smooth over a non-empty open subscheme $S\subset \Spec K[z_1,\ldots,z_m]$.
We put $A=\begin{bmatrix}
\ba_1&\cdots&\ba_m    
\end{bmatrix}$ and assume that the rank of $A$ is $n$.
Note that this assumption is automatically satisfied by the assumption that $F$ is generically smooth if $d\geq3$.
Let $\P^{n-1}_S={\rm Proj}(S[x_1,\dots,x_n])$ and let $X\subset \P^{n-1}_S$ be the hypersurface defined by $F$. Put $U=\P^{n-1}_S\setminus X$,
\begin{equation}\label{eq:Omega}
\Omega:=\sum_{i=1}^n(-1)^ix_idx_1\wedge\cdots\wh{dx_i}\cdots\wedge dx_n,
\quad\text{  and  }\quad \frac{d\bx}{\bx}:=(x_1\cdots x_n)^{-1}\Omega.    
\end{equation}
Let $H^i_\dR(U/S)$ be the $i$-th relative de Rham cohomology group.
For any $\bc=(c_1,\dots,c_n)\in\Z_{>0}^n$ that satisfies $|\bc|:=c_1+\cdots+c_n=dr$ with $r\in \Z_{\geq1}$, 
put
\begin{equation}\label{eq:omega.bc}
\omega_{\bc}:=\frac{x^{\bc}}{F^r}\frac{d\bx}{\bx}\in\varGamma(U,\Omega^{n-1}_{U/S})
    \end{equation}
a rational form. The form $\omega_{\bc}$ defines a cohomology class in $H^{n-1}_\dR(U/S)$, which
we write by the same symbol for simplicity.
Let $\mathcal{D}_S$ denote the sheaf of rings of differential operators on $S$.
We endow $H^{n-1}_{\dR}(U/S)$ with a structure of a left $\mathcal{D}_S$-module via Gauss-Manin derivative \cite[Section 9.2]{Voisin}. 
Let $D_K=K\langle z_i,\partial_i;i=1,\dots,N\rangle$ denote the Weyl algebra on 
$\Spec K[z_1,\dots,z_m]$.
We define the GKZ ideal $H_A(\bc)$ of $D_K$ in the same way as in \S \ref{sect:GKZ system}. 
For a differential operator $P\in D_K$ and a local section $\omega\in H^{n-1}_{\dR}(U/S)$, we write $P\bullet \omega$ for the action via Gauss-Manin derivative.
\begin{theorem}[{\cite[Theorem 2.7]{GKZ}}]
$H_A(\bc)\bullet\omega_{\bc}=0\in H^{n-1}_\dR(U/S)$.
\end{theorem}
Suppose that $K=\C$. Thanks to the above theorem, one has a $\C$-linear map
\[
\Sol^T(H^{n-1}_\dR(U/S))\lra \Sol^T(M_A(\bc))
\]
for a regular triangulation $T$ of $A$, where
$\Sol^T(-)=\Hom_D(-,\O^{\text{an}}_z)$ denotes the space of local solutions (see \S \ref{sect:GKZ system sol}).
In particular, the period integrals of $U/S$ are solutions of the GKZ system $(A,\bc)$.
By Theorem \ref{thm:SST limits span the solution space} together with the fact that
$\bc$ is in the interior of ${\rm pos}(A)$,
we have the following theorem.
For a fixed $\bz\in S$, we write $U_{\bz}$ for the fiber at $\bz$.
Given a cycle $\delta\in H_{n-1}^B(U_{\bz};K)$, we all $\int_{\delta}\omega_{\bc}$.
By Ehresmann's fibration theorem, one can continuously deform the cycle $\delta$ as $\bz$ varies to define a locally defined holomorphic function.
We denote this function by $\int_{\delta}\omega_{\bc}$ by abuse of notation.
\begin{theorem}
The following assertions hold true:
\begin{enumerate}
    \item The dimension of $\Sol^T(M_A(\bc))$ equals $\vol(A)$;
    \item A period integral $\int_\delta\omega_{\bc}$
belongs to $\Sol^T(M_A(\bc))$.
\end{enumerate}
\end{theorem}

The residue map induces an isomorphism
\begin{equation}\label{eq:residule map}
    \Res: H_\dR^{n-1}(U/S)\overset{\cong}{\lra} 
    H_\dR^{n-2}(X/S)_{\mathrm{prim}}:=\mathrm{Ker}
    [H_\dR^{n-2}(X/S)\to H^n_\dR(\P^{n-1}/S)]
\end{equation}
onto the primitive part, and this is compatible with the homomorphism 
$H_B^{n-1}(U_{\bz},\Q)\to H^{n-2}_B(X_{\bz},\Q(-1))_{\mathrm{prim}}$ 
between Betti cohomology.
In what follows, we often identify $H_\dR^{n-1}(U/S)$
with $H_\dR^{n-2}(X/S)_{\mathrm{prim}}$, and also
the Betti homology
$H^B_{n-1}(U_{\bz},\Q)$ with
$H_{n-2}^B(X_{\bz},\Q(1))_{\mathrm{prim}}$, the dual space of 
$H^{n-2}_B(X_{\bz},\Q(-1))_{\mathrm{prim}}$.

\section{Periods of Fermat deformations}
\label{sect:Fermat}
\subsection{Periods of Fermat varieties}\label{sect:periods of Fermat varieties}
Let $X$ be a smooth projective hypersurface $X\subset \P^{n-1}$
of degree $d$
defined over $\C$.
Let $F\in \C[x_1,\dots,x_n]$ be the defining polynomial of $X$.
Let $J_F\subset \C[x_1,\dots,x_n]$ be the ideal generated by $\frac{\partial F}{\partial x_1},\dots,\frac{\partial F}{\partial x_n}$ and let 
$R_F=K[x_1,\dots,x_n]/J_F$ be the quotient ring equipped with the 
natural grading $\{ R^p_F\}_{p\in\Z}$.
Let $F^\bullet H_\dR^{n-1}(\P^{n-1}\setminus X)$ denote the Hodge filtration on $H_\dR^{n-1}(\P^{n-1}\setminus X)$.
By \cite[\S 4]{Griffiths}, it follows that the arrow
\begin{equation}\label{eq:Griffiths reduction}
    R^{d(p-n+2)-n}_F\lra {\rm Gr}^p_FH_\dR^{n-1}(\P^{n-1}\setminus X)
    \cong {\rm Gr}^{p-1}_FH_\dR^{n-2}(X)_{\mathrm{prim}},
    \quad P\longmapsto \frac{P}{F^{p-n+2}}\Omega
\end{equation}
is bijective for $1\leq p\leq n-1$.
The Fermat variety $X_\Fer$ is the hypersurface defined by
the equation  
$$F_d:=x_1^d+\cdots+x_n^d.$$
Put $U_{\rm Fer}:=\mathbb{P}^{n-1}\setminus X_{\rm Fer}$.
Let $\mu_d:=\{ e^{2\pi\ii \frac{i}{d}}\mid i=0,1,\dots,d-1\}$ denote the cyclic group of order $d$ and let $D_d:=\{ (\zeta,\dots,\zeta)\in\mu_d^n\mid \zeta\in\mu_d\}$ denote the diagonal subgroup of $\mu_d^n$.
The quotient group $G_{\Fer}=\mu_d^n/D_d$ acts on $X_{\Fer}$ in a natural way as follows: for any $(\zeta_1,\dots,\zeta_n)\in G_{\Fer}$ and $[x_1:\cdots:x_n]\in X_{\Fer}$, the action is given by
\begin{equation}\label{eq:action}
    (\zeta_1,\ldots,\zeta_n)\bullet[x_1:\dots:x_n]=[\zeta_1x_1:\dots:\zeta_nx_n].
\end{equation}

For $\bc\in \Z^n$, let $H_\dR(\bc)$ denote the eigenspace 
of 
$H^{n-2}_\dR(X_{\Fer})_{\mathrm{prim}}$ on which $(\zeta_1,\ldots,\zeta_n)\in G_{\Fer}$ acts by multiplication by $\zeta_1^{c_1}\cdots\zeta_n^{c_n}$.
We set
\begin{equation}\label{eq:I}
    I:=\{\bc=(c_1,\dots,c_n)\in \Z^n\mid 0<\forall\,c_i<d,\,|\bc|\equiv 0\bmod d\}.
\end{equation}
Note that the cardinality of $I$ is given by $\frac{(d-1)((d-1)^{n-1}-(-1)^{n-1})}{d}$.
In view of the isomorphism \eqref{eq:Griffiths reduction}, we obtain the following simultaneous eigenspace decompositions:
\begin{equation}\label{Fermat-eq1}
H^{n-2}_\dR(X_{\Fer})_{\mathrm{prim}}=
\bigoplus_{\bc\in I}H_\dR(\bc),\quad \dim_{\ol\Q} H_\dR(\bc)=1
\end{equation}
and
\begin{equation}\label{Fermat-eq2}
H^{p,n-2-p}_\dR(X_\Fer)_{\mathrm{prim}}=
\bigoplus_{\bc\in I,\,|\bc|=(p+1)d}H_\dR(\bc).
\end{equation}
For $\bc\in \Z_{\geq 1}^n$ that satisfies $|\bc|=rd$, let
\begin{equation}
\omega_{\bc,\Fer}:=\frac{\bx^{\bc}}{\Fereq^r}\frac{d\bx}{\bx}=\frac{x_1^{c_1-1}\cdots x_n^{c_n-1}}{\Fereq^r}\Omega\in H^{n-1}_\dR(U_{\Fer})
\end{equation}
be the element as in \eqref{eq:omega.bc}.
The following lemma is used in the next section.

\begin{lemma}\label{lem:Fermat-cycle}
    There exists a cycle $[\delta_0]\in H^B_{n-1}(U_\Fer,\Q)$ such that for any positive vector $\bc\in\mathbb{Z}^n_{>0}$ with $|\bc|=dr$, the identity
    \begin{equation}\label{eq:3.7}
        \int_{\delta_0}\omega_{\bc,\Fer}=(2\pi\ii)^n\frac{(-1)^r}{\Gamma(r)\prod_{i=1}^n\Gamma(1-\frac{c_i}{d})}
    \end{equation}
    holds true.
\end{lemma}

\begin{proof}
Let $\Delta_{\bt}=\{(t_1,\cdots,t_n)\in\R^n\mid
t_i\geq0, \,t_1+\cdots+t_n=1\}$ be the standard $(n-1)$-simplex.
We consider an embedding $\Delta_{\bt}\ni
(t_1,\ldots,t_n)\mapsto
[t^{1/d}_1:\cdots:t^{1/d}_n]\in U_{\Fer}$, which defines a homology cycle
$[\Delta_{\bt}]\in H_{n-1}^B(U_{\Fer},D,\Z)$ where $D:=\{x_1\cdots x_n=0\}\cap U_{\Fer}$.
Recall that the group $G_{\Fer}=\mu_d^n/D_d$ acts on $U_{\Fer}$ and $D$.
Let $\sigma_i:=(1,\ldots,\zeta_d,\ldots,1)\in G_{\Fer}$ where $\zeta_d=e^{\frac{2\pi\ii}d}$ is
placed in the $i$-th component.
We put $\Theta:=\prod_{i=1}^n(\sigma_i-1)\in \Z[G_{\Fer}]$.
For each component $D_i=\{x_i=0\}\cap U_{\Fer}$, the action of $\sigma_i$ on $H^B_*(D_i,\Z)$ is the identity, and hence $\sigma_i-1$ acts on it by zero.
Thanks to the spectral sequence $E_1^{pq}=\bigoplus H^q_B(D_{i_0}\cap\cdots\cap D_{i_p},\Q)\Rightarrow H^{p+q}_B(D,\Q)$, it follows that $\Theta$ is nilpotent on $H^B_*(D,\Q)$, and hence is zero
since $\Q[G_{\Fer}]$ is a semisimple algebra.
Thus, the operator
$\Theta$ defines a homology class 
$\delta_0:=d^{n-1}\Theta[\Delta_{\bt}]\in H_{n-1}^B(U_{\Fer},\Q).
$
We have
\begin{align*}
\int_{\delta_0}\omega_{\bc,\Fer}
=&
\prod_{i=1}^n(\zeta_d^{c_i}-1)\cdot\int_{\Delta_{\bt}}t_1^{\frac{c_1}d}\cdots t_{n-1}^{\frac{c_{n-1}}d-1}
t_n^{\frac{c_n}d-1}\frac{dt_1}{t_1}\cdots \frac{dt_{n-1}}{t_{n-1}}\\
=&
\prod_{i=1}^n(\zeta_d^{c_i}-1)\cdot\int_{t_i\geq0,\sum_{i=1}^{n-1}t_i\leq 1}t_1^{\frac{c_1}d-1}\cdots t_{n-1}^{\frac{c_{n-1}}d-1}
(1-t_1-\cdots-t_{n-1})^{\frac{c_n}d-1}dt_1\cdots dt_{n-1}\\
=&
\prod_{i=1}^n(\zeta_d^{c_i}-1)
\frac{\Gamma({c_1}/d)\cdots \Gamma({c_n}/d)}{\Gamma(|\bc|/d)}.
\end{align*}
The last equality is the Dirichlet integral
(e.g. \cite[5.14.2]{olver2010nist}).
One can now derive \eqref{eq:3.7} from the formula $\Gamma(s)\Gamma(1-s)=\pi/\sin(\pi s)$.
\end{proof}

\subsection{Preliminaries on Fermat deformations}\label{sect:Fermat deformation}
We consider a homogeneous polynomial 
\begin{equation}\label{eq:Fermat deformation}
    F=F(\bx;\bz)
    :=\sum_{i=1}^mz_i\bx^{\ba_i}+z_{m+1}x_1^d+\cdots+z_{m+n}x_n^d,
\end{equation}
of degree $d$ 
where $z_1,\dots,z_{m+n}$ are deformation parameters.
Let $S\subset\Spec\C[z_1,\ldots,z_{m+n}]$ be an affine open set such that
the hypersurface $X\subset \P^{n-1}_S$ 
defined by $F(\bz;\bx)$ is smooth over $S$.
We call the family $X/S$ a {\it Fermat deformation}.
Put $U=\P^{n-1}_S\setminus X$.
As in \S \ref{sect:GKZ-Hypersurface}, 
let $A=\begin{bmatrix}
    \ba_1&\cdots&\ba_m&d\be_1&\cdots&d \be_n
\end{bmatrix}$ be the corresponding matrix to $F$ where $\be_1,\ldots,\be_n$ are the standard vectors.
We assume that $S$ contains the locus 
\begin{equation}\label{eq:the torus}
    \Sigma=\{z_1=\cdots =z_m=0
,\,z_{m+1}\cdots z_{m+n}\ne0\}.
\end{equation}
Let $G_{\Fer}=\mu^n_d/D_d$ be the group in \S \ref{sect:periods of Fermat varieties}.
Let $G\subset G_{\Fer}$ be the subgroup such that 
$\boldsymbol{\zeta}=(\zeta_1,\ldots,\zeta_n)$ belongs to $G$
if and only if $\boldsymbol{\zeta}^{\ba_i}=1$ for all $i$.
Then $G$ acts on $X$ in the way of \eqref{eq:action}.

\begin{lemma}\label{lem:FermatDeform.lemma}
Let $\ol\ba\in(\Z/d\Z)^n$ denote the image of $\ba\in\Z^n$. 
\begin{itemize}
\item[\rm(1)]
For any $\bc\in (\R_{>0})^n$, the holonomic rank of
$M_A(\bc)$ is the normalized volume $\vol(A)$.
\item[\rm(2)]
Let $\langle\ol\ba_1,\ldots,\bar\ba_m\rangle\subset(\Z/d\Z)^n$ denote the subgroup generated by
$\ol\ba_1,\ldots,\ol\ba_m$. 
Then one has
$
\vol(A)=\sharp\langle\ol\ba_1,\ldots,\bar\ba_m\rangle.
$
\end{itemize}
\end{lemma}
\begin{proof}
Notice that ${\rm pos}(A)=(\R_{\geq0})^n$.
Then, (1) follows from Theorem \ref{thm:SST limits span the solution space}.
The second assertion (2) follows from
the fact that
$[\Z^n:\Z A]=[(\Z/d\Z)^n:\langle\ol\ba_1,\ldots,\bar\ba_m\rangle]$
and the definition of $\vol(A)$ in
\eqref{eq:normalized volume}.
\end{proof}

Let $I$ be the set of
tuples in \eqref{eq:I}.
For any $\bc,\bc'\in I$, we define an equivalence relation $\bc\sim\bc'$ by the condition $\bar\bc'-\bar\bc\in\Z\bar\ba_1+\cdots+\Z\bar\ba_m$, where $\bar\ba\in(\Z/d\Z)^n$ denotes the reduction modulo $d$.
We set 
\begin{equation}\label{eq:definition:Ic}
I_{\bc}:=\{\bc'\in I\mid \bc\sim\bc'\}.
\end{equation}
Recall from \eqref{eq:omega.bc}
the element
\[
\omega_{\bc}:=
\frac{\bx^{\bc}}{F^r}\frac{d\bx}{\bx}\in H^{n-1}_\dR(U/S)
\]
of the de Rham cohomology for
$\bc\in \Z^n_{\geq1}$ that satisfies 
$|\bc|=rd$ with $r\in\Z_{\geq1}$.
We put $R:=\O(S)$ and $R_\Sigma:=\varinjlim\O_S(V)$ where 
$V$ runs over Zariski open sets which contain $\Sigma$ (see \eqref{eq:the torus} for the definition of $\Sigma$).

\begin{lemma}\label{lem:Fermat-dR-lemma}
The following assertions hold true:
\begin{itemize}
\item[\rm(1)]
$H^{n-1}_\dR(U/S)\otimes_RR_\Sigma$ has an $R_\Sigma$-basis $\{\omega_{\bc}\}_{\bc\in I}$.
In particular, 
$\operatorname{rank}H^{n-1}_\dR(U/S)\ot_RR_\Sigma=\sharp I$.
\item[\rm(2)]
For $\bc\in I$, let $H_\dR(\bc)$ denote the submodule of $H^k_{\dR}(U/S)$ on which 
$\boldsymbol{\zeta}\in G$ acts by
multiplication by $\zeta^{\bc}$.
Then $H_\dR(\bc)\ot_RR_\Sigma$ has an $R_\Sigma$-basis $\{\omega_{\bc'}\}_{\bc'\in I_{\bc}}$. 
In particular, 
$\operatorname{rank}
H_\dR(\bc)\ot_RR_\Sigma=\sharp I_{\bc}$.
\end{itemize}
\end{lemma}

\begin{proof}
We know that all $H_\dR(\bc)$ are locally free $R$-modules, so that
it suffices to show the lemma at each geometric points
in $\Sigma$.
However, at any point of $\Sigma$, 
the hypersurface $X$ is isomorphic to the Fermat variety.
It is simple to compute the basis of
the de Rham cohomology of Fermat varieties, and then one easily obtains the lemma 
(cf. the proof of \cite[Lemma 3.1]{katz2010another}).
\end{proof}

To an arbitrary $\bc\in I$, we associate a character of
$G$ given by $\boldsymbol{\zeta}\mapsto\boldsymbol{\zeta^{\bc}}$.
Then the set of characters of $G$ 
has one-to-one correspondence to the set $I/\sim$.
Let $\{\bc_i\}_i$ be a complete set of representatives of
$I/\sim$. 
We have a decomposition $I=\coprod_{i} I_{\bc_i}$, 
and a direct sum decomposition 
\begin{equation}\label{eq:decomposition of H_dR}
H^{n-1}_\dR(U/S)=\bigoplus_{i}H_\dR(\bc_i)
\end{equation}
as $R$-modules.
The decomposition \eqref{eq:decomposition of H_dR} is compatible with the Gauss-Manin derivative, hence, is a direct sum as $\mathcal{D}_S$-modules.

\medskip

The following lemma is due to Ryo Negishi.

\begin{lemma}[Negishi]\label{lem:Negishi-thm}
Let $\bc\in I$ be arbitrary.
Let $M_A(\bc)_S=\mathcal{D}_S/\mathcal{D}_SH_A(\bc)$ denote the left $\mathcal{D}_S$-module defined by
the GKZ-ideal.
Then the homomorphism
\[
R_\Sigma\ot_R M_A(\bc)_S\lra R_\Sigma\ot_R H_\dR(\bc),\quad P\longmapsto P\bullet\omega_{\bc}
\]
of left $\mathcal D_S$-modules is surjective.
\end{lemma}
\begin{proof}
Fix a point $\boldsymbol{0}\in T$.
For any $\bc'\in I_{\bc}$, we choose $p_i\geq 0$ so that a relation
\[
\bc'\equiv \bc+p_1\ba_1+\cdots+p_m\ba_m\mod d
\]
holds true. 
We put $\bc^{\prime\prime}:=
\bc+p_1\ba_1+\cdots+p_m\ba_m$ and $\bc^{\prime\prime}-\bc'
=(ds_1,\ldots,ds_n)$. Since $0<c'_i<d$, one has $ds_i>-d$ and hence $s_i\in \Z_{\geq0}$ for all $i$.
Let $\partial_i:=\partial/\partial t_i$.
Let $(-)|_{\boldsymbol 0}$ denote the specialzation at the point $\boldsymbol0$.
To prove the lemma, it is enough to show a relation
\begin{equation}\label{eq:a relation in de Rham}
\left(\prod_{i=1}^{m}\partial_i^{p_i}\bullet\omega_{\bc}-q\,\omega_{\bc'}\right)\bigg|_{\boldsymbol 0}
\equiv0
\end{equation}
with some $q\in\Q^\times$ 
in $H^{n-1}_\dR(U_{\boldsymbol0})$.
First, the following equalities hold:
\[
\prod_{i=1}^m\partial_i^{p_i}\bullet\omega_{\bc}=
\prod_{i=1}^m\partial_i^{p_i}\bullet\frac{\bx^{\bc}}{F^r}\frac{d\bx}{\bx}=
e\frac{\bx^{\bc+p_1\ba_1+\cdots+p_m\ba_m}}{F^{r+p_1+\cdots+p_m}}\frac{d\bx}{\bx}
=e\frac{\bx^{\bc^{\prime\prime}}}{F^{r+p_1+\cdots+p_m}}\frac{d\bx}{\bx},
\]
where $e:=(-1)^{p_1+\cdots+p_m}{(r+p_1+\cdots+p_m-1)!}/{(r-1)!}$.
Recall an equality
\[
k\frac{\bx^{\ba}\frac{\partial F}{\partial x_i}}{F^{k+1}}\Omega
= 
a_i\frac{\bx^{\ba}x^{-1}_i}{F^k}\Omega
\]
in $H^{n-1}_\dR(U/S)$, and hence
\[
kd\frac{\bx^{\ba}x_i^{d-1}}{F^{k+1}}\Omega
\bigg|_{\boldsymbol 0}
\equiv 
a_i\frac{\bx^{\ba}x^{-1}_i}{F^k}\Omega\bigg|_{\boldsymbol 0}.
\]
Using this repeatedly, we obtain
\begin{align*}
\frac{\bx^{\bc^{\prime\prime}}}{F^{r+p_1+\cdots+p_m}}\frac{d\bx}{\bx}\bigg|_{\boldsymbol 0}
&=
\frac{\bx^{\bc^{\prime}}x_1^{ds_1-1}\cdots x_n^{ds_n-1}}{F^{r+p_1+\cdots+p_m}}\Omega\bigg|_{\boldsymbol 0}
\equiv
\frac{\bx^{\bc^{\prime}}x_1^{d(s_1-1)-1}\cdots x_n^{ds_n-1}}{F^{r+p_1+\cdots+p_m-1}}\Omega\bigg|_{\boldsymbol 0}\\
&\equiv\cdots\equiv
\frac{\bx^{\bc^{\prime}}}{F^{r+p_1+\cdots+p_m-(s_1+\cdots+s_m)}}\frac{d\bx}{\bx}\bigg|_{\boldsymbol 0}
=\omega_{\bc'}|_{\boldsymbol 0}
\end{align*}
up to a scalar multiplication by $\Q^\times$. This completes the proof of \eqref{eq:a relation in de Rham}.
\end{proof}

\subsection{Periods of Fermat deformations}
\label{sect:periods of Fermat deformations}
Let $T$ be a regular triangulation of $A$.
By the decomposition \eqref{eq:decomposition of H_dR}, there is a $\C$-linear map
\begin{equation}\label{eq:dR to GKZ}
\Sol^{T}(H^{n-1}_\dR(U/S))\os{\cong}{\lra}
\bigoplus_i\Sol^{T}(H_\dR(\bc_i))
\os{\subset}{\lra}\bigoplus_i\Sol^{T}(M_A(\bc_i))
\end{equation}
where the second injective arrow is induced from the homomorphism in Lemma \ref{lem:Negishi-thm}.
Composing it with the natural map $H_{n-1}^B(U_{\bz},\Q)\to\Sol^{T}(H^{n-1}_\dR(U/S))$, one has
\begin{equation}\label{eq:betti}
H^B_{n-1}(U_{\bz},\Q)\os{\subset}{\lra} \bigoplus_i\Sol^T
(M_A(\bc_i)),\quad \delta\longmapsto\left(\int_\delta
\omega_{\bc_i}\right)_i,
\end{equation}
where $U_{\bz}$ denotes the fiber at a point $\bz\in S$ (possibly depends on $T$).
In this section, we compute a $\Q$-basis of the image of the map \eqref{eq:betti}.

\medskip

Let us fix a generic weight vector $\bw=(w_1,\ldots,w_{n+m})\in\R^{m+n}$ such that $w_1,\ldots,w_m\gg\max(w_{m+1},\ldots,w_{m+n})$.
Set $T({\rm Fer}):=S(\bw)$, 
the regular triangulation of $A$ which has a single simplex $\sigma=\{ m+1,\dots,m+n\}$.

For $\bc\in\Q^n_{>0}$ and $\bp=(p_1,\dots,p_m)\in\Z^m$, set
\begin{equation}\label{eq: gamma.bc.bp}
\bg_{\bp}^{\bc}=\begin{bmatrix}
p_1\\ \vdots \\ p_m\\ -\frac1d\bc-\frac1d(p_1\ba_1+\cdots+p_m\ba_m)
\end{bmatrix}.
\end{equation}
This is the unique solution to the equation \eqref{eq:char. exponent} with $\sigma=\{m+1,\dots,m+n\}$.
Let $h:\Z^m\to (\Z/d\Z)^n$ be the homomorphism given by
\begin{equation}\label{eq:morphism h}
    h:\Z^m\ni\bq\mapsto \sum_{i=1}^mq_i\bar\ba_i\in(\Z/d\Z)^n,
\end{equation}
where $\ol\bu\in(\Z/d\Z)^n$ denotes the reduction modulo $d$ for $\bu\in\Z^n$ and let
$\Phi_{T(\Fer)}$ be the set \eqref{eq:series solutions} corresponding to $T=T(\Fer)$.

\begin{lemma}\label{lem:lem-bar a}
If $\bc$ is very generic,
then the map
$\Z^m/\mathrm{Ker}(h)\to \Phi_{T(\Fer)}$ given by $\bq\mapsto
\varphi(\bg_{\bq}^{\bc};\bz)$ is well-defined and bijective.
In particular, $\{\varphi(\bg_{\bq_i}^{\bc};\bz);1\leq i\leq \ell\}$
forms a $\C$-basis of $\Sol^{T(\Fer)}(M_A(\bc))$ for a representative $\bq_1,\ldots,\bq_\ell\in\Z^m$
of $\Z^m/\mathrm{Ker}(h)$ by Theorem \ref{thm:eq:condition *}.
\end{lemma}
\begin{proof}
If $\bq-\bq'\in\mathrm{Ker}(h)$, then $\bg_{\bq}^{\bc}-\bg^{\bc}_{\bq'}
\in L_A={\rm Ker}_\Z(A)$, and so $\varphi(\bg_{\bq}^{\bc};\bz)=\varphi(\bg_{\bq'}^{\bc};\bz)$.
Therefore, one has a well-defined map
$\phi:\Z^m/\mathrm{Ker}(h)\to \Phi_{T(\Fer)}$ given by $\bq\mapsto
\varphi(\bg_{\bq}^{\bc};\bz)$, and it is onto
by definition of $\Phi_{T(\Fer)}$.
Let
$\Phi_{T(\Fer)}$ be the set \eqref{eq:series solutions}.
Since $\bc$ satisfies the condition in Theorem \ref{thm:eq:condition *}
and $\bc\in{\rm pos}(A)$,
the set $\Phi_{T(\Fer)}$ forms a $\C$-basis of $\Sol^{T(\Fer)}(M_A(\bc))$, and 
hence the cardinality $\sharp\Phi_{T(\Fer)}$ equals $\vol(A)$ 
by Theorem \ref{thm:SST limits span the solution space}.
On the other hand, the cardinality  
$\sharp(\Z^m/{\mathrm{Ker}}(h))=\sharp\langle\bar\ba_1,\ldots,\bar\ba_m\rangle$ also
equals $\vol(A)$ by Lemma \ref{lem:FermatDeform.lemma}.
Therefore the map $\phi$ is bijective.
\end{proof}

Put $\RNAfK:=\pi_1({\rm Log}^{-1}(r\bw_T),\bz)$ the fundamental group of the real torus for $T=T(\Fer)$
as in \eqref{eq:local monodromy}, and let
\begin{equation}\label{eq:local monodromy Fer}
\RNAfK\lra \mathrm{GL}(\Sol^{T(\Fer)}({\mathcal M}))
\end{equation}
be the local monodromy on the solution space.
Suppose that 
$\bc\in\R_{>0}^n$ is very generic (i.e. the condition in Theorem \ref{thm:eq:condition *}).
Let $Q_i\in \RNAfK$
be the counter-clockwise loop around the coordinate $z_{m+i}$.
Let
$\Phi_{T(\Fer)}=\{\varphi(\bg_{\bq_i}^{\bc};\bz)\}_i$
be the $\C$-basis of 
$\Sol^{T(\Fer)}(M_A(\bc))$ 
in Lemma \ref{lem:lem-bar a}.
For each $\bq\in\Z^m$,
the monodromy operator $Q_{k}$ acts on
$\varphi(\bg_{\bq}^{\bc};\bz)$
by multiplication by
$\exp(2\pi i\,u_k/d)$ where $u_k$ is the $k$-th component of
$-\bc-(q_1\ba_1+\cdots+q_m\ba_m)$.
Therefore, any simultaneous eigenspace of $\Sol^{T({\Fer})}(M_A(\bc))$ with respect to
the group $\RNAfK$ is one-dimensional by Lemma \ref{lem:lem-bar a}.
When $\bc\in\Q^n_{>0}$ is arbitrary, 
let $\bc(\varepsilon):=\bc+\varepsilon\bv$ with $\bv\in\Q^n$ a generic vector.
Then it follows from Theorem \ref{thm:SST limits span the solution space} that
the SST-limit gives the $\vol(A)$-dimensional solutions,
\[
\lim_{\varepsilon\to0}\Sol^{T(\Fer)}(M_A(\bc(\varepsilon)))=\Sol^{T(\Fer)}(M_A(\bc)).
\]
Every simultaneous eigenspace of $\Sol^{T({\rm Fer})}(M_A(\bc(\varepsilon)))$ are also one-dimensional,
and hence 
the limits
\begin{equation}\label{eq:SST-series solution}
\lim_{\varepsilon\to0}\varepsilon^{k_i}
\varphi(\bg_{\bq_i}^{\bc(\varepsilon)};\bz),\quad(1\leq i\leq \ell)
\end{equation}
with $k_i\in \Z$ suitably chosen
give a $\C$-basis of the SST-limit (Remark \ref{remark:SST-limit defn}).
In particular, the basis \eqref{eq:SST-series solution} does not involve the logarithmic function.

\medskip

Summing up the above, we have the following lemma.
\begin{lemma}\label{lem:multiplicity.one}
Let $\bc\in \Q^n_{>0}$ be arbitrary.
Then, 
every simultaneous eigenspace of $\Sol^{T({\rm Fer})}(M_A(\bc))$ with respect to
$\RNAfK$ is one-dimensional, and
its basis is spanned by (log free) series solutions
\eqref{eq:SST-series solution}.
\end{lemma}

For $1\leq i\leq m$, we set $\bb_i:=(    0 , \dots, d , \dots,0 , -a_{1i},\dots,-a_{ni})\in\Z^{m+n}$
where the number $d$ is in the $i$-th entry.
Let $\tilde S=\Spec R[z_{m+1}^{1/d},\ldots,z_{m+n}^{1/d}]\to S$ 
be a covering map, and let $\tilde X\to\tilde S$ 
(resp. $\tilde U\to \tilde S$) 
be the pull-back of the family $X\to S$ (resp. $U\to S$).
One sees that $\tilde X$ is given by an equation
\[
\sum_{i=1}^m\bz^{\frac{\bb_i}d}\bx^{\ba_i}+
x_1^d+\cdots+x_n^d=0,
\]
and the family $\tilde X\to\tilde S$ has a good reduction over a locus $z_{1}=\cdots=z_m=0$.
Since $\tilde X\to \tilde S$ is a topological fibration,
one has the deformation
$\delta_{\bz}\in H^B_{n-1}(U_{\bz},\Q)$ 
of the homology cycle $\delta_0$ constructed in Lemma \ref{lem:Fermat-cycle}
for each point $\bz\in S$ such that $0<|z_i|\ll |z_j|$ for all $1\leq i\leq m$ and for all $m+1\leq j\leq m+n$.

Now, we state the key lemma of this section.
To do so, we introduce some notations.
Let $\bc\in (\Z_{>0})^n$ be an arbitrary vector such that $|\bc|=dr$ for some $r\in\Z$.
Let $h:\Z^m\to (\Z/d\Z)^n$ be the morphism \eqref{eq:morphism h}, and 
we fix representatives $\bq_1,\ldots,\bq_\ell\in \Z^m$ of the set $\Z^m/\mathrm{Ker}(h)$.
We define a subset  $P_{\bc}\subset\{\bq_1,\ldots,\bq_\ell\}$ by the following condition:
\begin{equation}\label{eq:P_c}
\text{
 $\bp\in P_{\bc}$ if and only if
no entry of 
$\bar\bc+h(\bp)\in(\Z/d\Z)^n$ is zero.
}
\end{equation}

\begin{lemma}\label{lem:key-lem} 
Under the notation above, one has an expansion
\begin{equation}\label{eq:key-lem}
\int_{\delta_{\bz}}\omega_{\bc}=\frac{(-1)^r}{(r-1)!}(2\pi\ii)^n
\sum_{\bp\in P_{\bc}}\varphi(\bg_{\bp}^{\bc};\bz).
\end{equation}
\end{lemma}
\begin{proof}
By a change of coordinates $x_i=z_{m+i}^{-1/d}y_i$, the left-hand side of \eqref{eq:key-lem} reads
\begin{equation}\label{eq:y-integral}
\int_{\delta_{\bz}} \frac{z_{m+1}^{-\frac{c_1}{d}}\dots z_{m+n}^{-\frac{c_n}{d}}\,\by^{\bc-1}}{(y_1^d+\cdots+y_n^d+
\sum_{i=1}^m\bz^{\frac1d\bb_i}\by^{\ba_i})^r}\Omega_y,    
\end{equation}
where $\Omega_y:=\sum_{i=1}^n(-1)^iy_idy_1\wedge\cdots\widehat{dy_i}\cdots \wedge dy_n$.
In view of the expansion
\[
    (y_1^d+\cdots+y_n^d+
\sum\bz^{\frac1d\bb_i}\by^{\ba_i})^{-r}=\sum_{\bp\in \Z_{\geq 0}^m}(-1)^{|\bp|}
\frac{\Gamma(r+|\bp|)}{\Gamma(r){p_1!\cdots p_m!}}\prod_{i=1}^m\bz^{\frac{p_i}{d}\bb_i}\frac{\by^{\sum_{i=1}^mp_i\ba_i}}
{(y^d_1+\cdots+y^d_n)^{r+|\bp|}}
\]
and Lemma \ref{lem:Fermat-cycle}, we obtain that the integral \eqref{eq:y-integral} equals
\begin{align*}
    &\frac{(-1)^r}{(r-1)!}(2\pi\ii)^n\prod_{i=1}^nz_{m+i}^{-\frac{c_i}{d}}
\sum_{\bp\in\Z_{\geq0}^m}\frac{\prod_{i=1}^m\bz^{\frac{p_i}{d}\bb_i}}
{p_1!\cdots p_m!\Gamma\left(1-\frac{1}{d}(\bc+\sum_{i=1}^m p_i\ba_i)\right)}\\
    =&\frac{(-1)^r}{(r-1)!}(2\pi\ii)^n
\sum_{\bp\in\Z_{\geq0}^m}\frac{\bz^{\bg^{\bc}_{\bp}}}
{\Gamma(1+\bg^{\bc}_{\bp})}
    =\frac{(-1)^r}{(r-1)!}(2\pi\ii)^n
    \sum_{\bp\in\Z^m}
\frac{\bz^{\bg^{\bc}_{\bp}}}
{\Gamma(1+\bg^{\bc}_{\bp})}
\end{align*}
where $\bg^{\bc}_{\bp}=(\bp,-\frac1d(\bc+\sum p_i\ba_i))$ is the vector \eqref{eq: gamma.bc.bp}.
Since 
$\bg^{\bc}_{\bp}-\bg^{\bc}_{\bp'}\in L_A\iff \bar\bc+h(\bp)=\bar\bc+h(\bp')$, one has
\[
\sum_{\bp\in\Z^m}
\frac{\bz^{\bg^{\bc}_{\bp}}}
{\Gamma(1+\bg^{\bc}_{\bp})}=
\sum_{i=1}^\ell\sum_{\bu\in L_A}
\frac{\bz^{\bg^{\bc}_{\bq_i}+\bu}}
{\Gamma(1+\bg^{\bc}_{\bq_i}+\bu)}=\sum_{i=1}^\ell
\varphi(\bg_{\bq_i}^{\bc};\bz).
\]
If some entry of $\bar\bc+h(\bp)$ is zero, then $1/{\Gamma(1+\bg^{\bc}_{\bp})}=0$
and therefore $\varphi(\bg_{\bq_i}^{\bc};\bz)=0$ unless $\bq_i\in P_{\bc}$.
In summary, the integral \eqref{eq:y-integral} equals
the right hand side of \eqref{eq:key-lem}.
\end{proof}

We note that the set $\{\varphi(\bg_{\bp}^{\bc};\bz);\bp\in P_{\bc}\}$ does not depend on the choice of
the representatives $\bq_1,\ldots,\bq_\ell$
by Lemma \ref{lem:lem-bar a}.

\begin{lemma}\label{lem:image of the integration map}
Let $G$ be the finite abelian group in the beginning of \S \ref{sect:Fermat deformation}.
Let $\langle \RNAfK,G\rangle$ be 
the subgroup of $\mathrm{GL}(H^B_{n-1}(U_{\bz},\Q))$
generated by the images of $\RNAfK$ and $G$, and let
$\Q[\RNAfK,G]$ denote its group ring.
Note $\langle \RNAfK,G\rangle$ is commutative.
Then
\[H^B_{n-1}(U_{\bz},\Q)=\Q[\RNAfK,G]\cdot\delta_{\bz}.
\]
 \end{lemma}
\begin{proof}
Recall \eqref{eq:dR to GKZ}, which induces 
\[
H_{n-1}^B(U_{\bz},\C)\os{\cong}{\lra}
\bigoplus_i\Sol^{T(\Fer)}(H_{\dR}(\bc_i))
\subset\bigoplus_i
\Sol^{T(\Fer)}(M_A(\bc_i)).
\]
By Lemma \ref{lem:Fermat-dR-lemma} (1),
it is enough to show $\dim\C[\RNAfK,G]\cdot\delta_{\bz}=\sharp I$.
To see this, we put 
\[
W_i:=\C[\RNAfK]\cdot\int_{\delta_{\bz}}\omega_{\bc_i}
\subset 
\Sol^{T(\Fer)}(M_A(\bc_i)).
\]
The image of $\C[\RNAfK,G]\cdot\delta_{\bz}$ is the direct sum
of $W_i$'s (in view of the action of $G$).
Therefore it is enough to show $\dim W_i\geq \sharp I_{\bc_i}$
by Lemma \ref{lem:Fermat-dR-lemma} (2).

Let $P_{\bc_i}=\{\bp_1,\ldots,\bp_s\}$ be the set as in Lemma \ref{lem:key-lem}.
We note $s=\sharp I_{\bc_i}$ by definition of $I_{\bc_i}$ in \eqref{eq:definition:Ic}.
Let $Q_j\in \RNAfK$ be
the monodromy operator around the coordinate $z_{m+j}$. 
This acts on 
the series $\varphi(\bg_{\bp_i}^{\bc};\bz)$ by a scaler multiplication by
$$
\alpha_{\bp_i,j}:=\exp\left(\frac{2\pi\ii}{d}(c_j+p_{i,1}a_{1j}+\cdots+p_{i,m}a_{mj})\right).
$$
Hence the operator
$$
\widehat{Q}_i:=\prod_{j=1}^m(Q_j-\alpha_{\bp_i,j})\in \C[\RNAfK]
$$
annihilates any $\varphi(\bg_{\bp_{i'}}^{\bc};\bz)$ ($i'\neq i$), and 
acts on $\varphi(\bg_{\bp_i}^{\bc};\bz)$ by a non-zero scalar multiplication.
By \eqref{eq:key-lem}, it turns out that $W_i$
contains $\varphi(\bg_{\bp_i}^{\bc};\bz)$ for every $1\leq i\leq s$, and hence
$\dim W_i\geq s$ as required.
\end{proof}
The proof of Lemma \ref{lem:image of the integration map} also shows the following result.
\begin{theorem}\label{thm:image of the integral of omega_c}
Let $P_{\bc}$ be the set as in Lemma \ref{lem:key-lem}.
The image of the map
\begin{equation}\label{eq:integral of omega_c}
H_{n-1}^B(U_{\bz},\Q)\ot\Q(\zeta_d)\lra \Sol^{T(\Fer)}(M_A(\bc)),\quad
\delta\longmapsto\int_\delta\omega_{\bc}    
\end{equation}
has a $\Q(\zeta_d)$-basis
$\{(2\pi\ii)^{n}\varphi(\bg^{\bc}_{\bp};\bz)\}_{\bp\in P_{\bc}}$.    
\end{theorem}
\begin{theorem}\label{thm:arbitrary triangulation}
Let $T$ be a regular triangulation of $A$.
Then, there exists an algorithm to compute a $\Q$-basis of the image of the map $H_{n-1}^B(U_{\bz},\Q)\to\Sol^{T}(M_A(\bc))$ defined by \eqref{eq:betti}.
\end{theorem}
\begin{proof}
For $T=T(\Fer)$, it follows from Lemmas 
\ref{lem:key-lem} and \ref{lem:image of the integration map}.
For an arbitrary $T$,
there is a sequence of regular triangulations $T_1=T(\Fer),T_2,\ldots,T_\ell=T$ such that
$T_i$ and $T_{i+1}$ share a facet.
Then, for each $i$, we apply the connection formula (Theorem \ref{thm:connection formula})
for $T_i$ and $T_{i+1}$.
\end{proof}

\section{Limiting Periods of Degeneration of Hypersurfaces}
\label{sect:limitMHS}

\subsection{Review of limiting mixed Hodge structures}\label{sect:ReviewLimitMHS}
Following \cite[\S 11 and 14]{peters2008mixed}, we recall the 
limiting mixed Hodge structures for one-parameter degenerations.

\medskip

Let $X$ be a complex manifold, and 
$\rho:X\to \Delta=\{t\in \C \mid |t|<1\}$ be a projective flat morphism
over the disk which is smooth over $\Delta^*=\Delta\setminus\{0\}$.
Put $X^*=\rho^{-1}(\Delta^*)$.
We may or may not assume that the reduced part of the central
fiber is a simple normal crossing divisor in $X$, because it does not matter
for the construction of the limiting MHS.
Let $i\geq0$ be an integer, and put
\[
H_\dR:=R^i\rho_*\Omega^\bullet_{X^*/\Delta^*},\quad H_B:=R^i\rho_*\Q.
\]
Here, $H_\dR$ is endowed with the Hodge filtration 
$F^\bullet$ and
the action of $\partial_t:=d/dt$
by the Gauss-Manin connection.
Furthermore, $H_B$ is a local system on $\Delta^*$, so the cohomology  $H^i_B(X_{t_0},\Q)$ of $X_{t_0}=f^{-1}(t_0)$
is endwoed with the monodormy $T$ 
where we fix a point $t_0\in \Delta^*$.
Let $V_\bullet H_\dR$ be the $V$-filtration by Kashiwara-Malgrange.
This is the unique increasing filtration indexed by $\Q$
which is characterized by the following properties:
\begin{align*}
&tV_\alpha H_\dR\subset V_{\alpha-1}H_\dR;\\
&\partial_t V_\alpha H_\dR\subset V_{\alpha+1}H_\dR;\,\text{and}\\
&t\partial_t+\alpha+1\text{ is nilpotent on }\mathrm{Gr}^V_\alpha H_\dR,
\end{align*}
where $\mathrm{Gr}^V_\alpha H_\dR:=V_\alpha H_\dR/V_{<\alpha} H_\dR$.
Note that $V_{-1}H_\dR/V_{-2}H_\dR$ is exactly Deligne's canonical extension
of $(H_\dR,\nabla)$ (see \cite[Definition 11.4]{peters2008mixed}).
We put
\[H_{B,\infty}:=H^i_B(X_{t_0},\Q),\quad
H_{\dR,\infty}:=\bigoplus_{-1\leq \alpha<0}\mathrm{Gr}^V_\alpha H_\dR.\]
Note that $H_{\dR,\infty}$ is a sheaf on $\Delta$ supported at the origin.
Therefore, we also regard it as a vector space by identifying it as its stalk at the origin.
We define the weight filtration $W_\bullet H_{B,\infty}$ 
to be the monodromy
filtration by the unipotent part $T_u$ of $T$, namely the unique increasing filtration
characterized by the properties
\[
NW_\bullet H_{B,\infty}\subset 
W_{\bullet-2}H_{B,\infty}, \quad N^k:\mathrm{Gr}^W_{i+k}H_{B,\infty}\os{\cong}{\lra}
\mathrm{Gr}^W_{i-k}H_{B,\infty}
\]
for all $0\leq k\leq i$
where $N:=\log T_u$.
The Hodge filtration
$F^\bullet H_{\dR,\infty}$ is defined to be the image $\bigoplus_\alpha \Image(
F^\bullet 
H_\dR\cap V_\alpha H_\dR\to\mathrm{Gr}^V_\alpha H_\dR)$.
The comparison isomorphism
\begin{equation}\label{eq:comarison isomorphism}
\iota_\infty:H_{\dR,\infty}\os{\cong}{\lra} \C\ot_\Q H_{B,\infty}
=\Hom_\Q(H^B_i(X_{t_0},\Q),\C)
\end{equation}
is given in the following way (
\cite[Theorem 11.16]{peters2008mixed}, 
\cite[(2.6)]{steenbrink1976limits}).
For an element
\[g(t)=g_0(t)+g_1(t)\log t+\cdots+g_k(t)(\log t)^k\in\bigcup_{N\geq1}\C[\![t^{1/N}]\!][t^{-1},\log t],\]
let
$\mathbf{in}_\alpha(g(t))$ denote the coefficient of $t^\alpha$
in $g_0(t)$.
For $\omega\in H_{\dR}$ and $\delta\in H_i^B(X_{t_0},\Q)$,
one has an asymptotic expansion
\begin{equation}\label{eq:asymptotics}
\int_\delta\omega=\sum_{i\geq0} h_i(t)(\log t)^i    
\end{equation}
where $h_i(t)$ are holomorphic functions defined on a 
neighborhood of $t_0$ which have
Puiseux expansions with respect to $t$.
Observe that if $\omega\in V_\alpha H_\dR$, then
it imposes $\ord_t(h_i)\geq-\alpha-1$.
To an arbitrary element $\omega\in V_\alpha H_\dR$, 
we associate the homomorphism 
\begin{equation}\label{eq:iota infty}
 H_i^B(X_{t_0},\Q)\lra \C,\quad \delta\longmapsto \mathbf{in}_{-\alpha-1}\left(\int_\delta\omega\right),
\end{equation}
and this induces a linear map $\iota_{\infty,\alpha}:\mathrm{Gr}^V_\alpha H_\dR\to \C\ot_\Q H_{B,\infty}$.
We then define the comparison isomorphism \eqref{eq:comarison isomorphism}
by $\iota_\infty:=\bigoplus_{-1\leq \alpha<0}\iota_{\infty,\alpha}$.
The data
\begin{equation}\label{eq:nearby.cycle}
H^i_\infty(X/\Delta):=(H_{\dR,\infty},H_{B,\infty},F^\bullet H_{\dR,\infty},W_\bullet H_{B,\infty},\iota_\infty)
\end{equation}
forms a mixed Hodge structure, which we call the {\it limiting MHS}
for $R^if_*\Q$ (cf. \cite[Theorem 11.22]{peters2008mixed}).
By the construction, it is not affected by the central fiber,
namely depends only on $R^if_*\Q|_{\Delta^*}$.
On the other hand, it depends on the choice of the parameter $t$.

\medskip

Let $K\subset \C$ be a subfield. Suppose that $\rho:X\to\Delta$ has a
descent to $\Spec K[\![t]\!]$, which means that
there is a projective flat morphism $\rho_K:X_K\to \Spec K[\![t]\!]$
such that $\Spec\C[\![t]\!]\times_{\Spec K[\![t]\!]}X_K\cong \Spec\C[\![t]\!]
\times_\Delta X$.
Then $H_\dR$ is naturally endowed with a $K$-structure $H_{\dR,K}=R^i\rho_{K*}\Omega^\bullet_{X_K/K(\!(t)\!)}$
and the $V$-filtration is defined on $H_{\dR,K}$.
Let us put 
\[
H_{\dR,\infty,K}:=\bigoplus_{-1\leq \alpha<0}\mathrm{Gr}^V_\alpha H_{\dR,K}.
\]
Note that a relation $\C\ot_KH_{\dR,\infty,K}= H_{\dR,\infty}$ holds true.
We put
\begin{equation}\label{eq:nearby.cycle.K}
H^i_\infty(X/\Delta)_K:=(H_{\dR,\infty,K},H_{B,\infty},F^\bullet H_{\dR,\infty},W_\bullet H_{B,\infty},\iota_{\infty,K}).    
\end{equation}
One has the matrix representation of the comparison
\begin{equation}\label{eq:nearby.cycle.comparison}
\iota_{\infty,K}:\C\ot_K H_{\dR,\infty.K}\os{\cong}{\lra} \C\ot_\Q H_{B,\infty}
\end{equation}
with respect to a $K$-basis of $H_{\dR,\infty,K}$ and a $\Q$-basis of
$H_{B,\infty}$ as before.
We call the entries of the representation matrix of $\iota_{\infty,K}$ the
{\it limiting periods}.


\begin{remark}\label{rem:4.0}
When
$X/\Delta$ is a semistable reduction,
the limiting MHS
has an alternative construction by Steenbrink
\cite{steenbrink1976limits}, see also \cite[\S 11.2]{peters2008mixed}.
\end{remark}
\begin{remark}\label{rem:4.1}
The limiting MHS is stable under the base change
(cf. \cite[189, Proposition]{sabbah2022degenerating}).
To see this,  
let $\widetilde\Delta$ denote another disk with coordinate function $s$,
and let $i_n:\widetilde\Delta\to\Delta$ be the map given by 
$i_n^*(t)=s^n$ for an integer $n\geq1$.
Take $n$ such that
the base change $\widetilde\rho:\widetilde{X}\to 
\widetilde\Delta$ of $X/\Delta$
has a semistable reduction at $s=0$.
Let $\widetilde{H}_{\dR,K}$ be the de Rham cohomology of 
$\widetilde{X}/\widetilde\Delta$, and
let
\[
H^i_\infty(\widetilde{X}/\widetilde\Delta)_K=(\widetilde{H}_{\dR,\infty,K},H_{B,\infty},
F^\bullet \widetilde{H}_{\dR,\infty},W_\bullet H_{B,\infty},\widetilde\iota_{\infty}) 
\]
be the limiting MHS,
where the Betti cohomology $H_{B,\infty}$ together with the weight filtration $W_\bullet$
is identical with
the one in \eqref{eq:nearby.cycle.K}.
Then the map
\[
V_\alpha H_{\dR,K}\lra V_{-1}\widetilde{H}_{\dR,K},\quad \omega\longmapsto s^{n\alpha+n}\,i^*\omega
\]
induces a bijection 
\[
\widetilde{u}:
\bigoplus_{-1\leq \alpha<0}\mathrm{Gr}^V_\alpha H_{\dR,K}
\overset{\cong}{\lra}
\mathrm{Gr}^V_{-1}\widetilde{H}_{\dR,K},
\]
and this preserves the Hodge filtration 
and satisfies $\widetilde\iota_\infty\circ 
\widetilde{u}=\iota_\infty$.
Therefore, we have an isomorphism
$H^i_\infty(X/\Delta)_K\cong
H^i_\infty(\widetilde{X}/\widetilde\Delta)_K$ 
of mixed Hodge structures.
\end{remark}
\subsection{Main Theorem}\label{sect:Main Theorem}
Let $\P^{n-1}$ be the projective space 
with the coordinates $(x_1,\ldots,x_n)$, let $m>n$ be an integer, and let
\begin{equation}\label{F}
F=\sum_{i=1}^mz_i \bx^{\ba_i}
\end{equation}
be a homogeneous polynomial of degree $d$ where $z_1,\dots,z_m$ are deformation parameters.
Let $M=\Z^m$, $M'=\Z^n$ and 
we regard $A=\begin{bmatrix}\ba_1&\cdots&\ba_m \end{bmatrix}$ 
as a homomorphism $A:M\to M'$.

Let $\O_t^{an}$ denote the ring of holomorphic functions at $t=0$,
and let $\O_t^{an}[t^{-1}]$ denote the ring of meromorphic functions.
For $f=(f_1,\ldots,f_m)\in (\O^{an}_t[t^{-1}]\setminus\{0\})^m$, let $X_f\to \Delta$ be a projective flat family of 
hypersurfaces defined by an equation
\begin{equation}\label{F_f}
F_f=\sum_{i=1}^{m} f_i(t)\bx^{\ba_i}.
\end{equation}
Let $K\subset\C$ be a subfield.
We assume that both of the following conditions hold.
\begin{align}
&\text{$X_f$ is smooth over the punctured disk
$\Delta^*=\{t\in\C\mid0<|t|<1\}$;}\label{eq:condition (i)}\\
&\text{$f_i(t)\in K(\!(t)\!):=K[\![t]\!][t^{-1}]$ for all $i$.}\label{eq:condition (ii)}
\end{align}
By taking a suitable blowing up, we can 
assume that $X_f$ is a complex manifold, though it does not affect
the limiting MHS.
By \eqref{eq:condition (ii)}, $X_f$ has a descent to $\Spec K[\![t]\!]$ as in \S \ref{sect:ReviewLimitMHS}, 
so that one has the limiting MHS $H^{n-2}_\infty(X_f/\Delta)_K$ as in \eqref{eq:nearby.cycle.K}.
Let $\be_1,\ldots,\be_m$ be the standard basis of $M_\R$ and
$\be_1^*,\ldots,\be_m^*$ the dual basis of $M^*_\R$.
Let us put
\begin{equation}\label{eq:w-vec}
    \bw_f:=\ord_t(f_1)\be_1^*+\cdots+\ord_t(f_m)\be^*_m\in M_\R^*=(\R^m)^*,
\end{equation}
where $\ord_t(f_i)$ denotes the vanishing order of $f_i(t)$ at $t=0$.
Let $\bu_1,\ldots,\bu_{m-n}$ be a $\Z$-basis of $L_A$, let $\bu^*_1,\ldots,\bu^*_{m-n}$ be its dual basis and let $\pi_A:M_\R^*\to (L_{A,\R})^*$ be the morphism induced from the inclusion $L_A\hookrightarrow \Z^m$.
Then, one has an identity
\begin{equation}\label{eq:weight vector}
\pi_A(\bw_f)=\ord_t f^*(\bz^{\bu_1})\bu_1^*+\cdots+\ord_t f^*(\bz^{\bu_{m-n}})\bu_{m-n}^*
\in (L_{A,\R})^*.
\end{equation}
Here, we set $f^*(\bz^{\bu})=f_1^{u_1}\cdots f_m^{u_m}$.
It also follows that the right-hand side of \eqref{eq:weight vector} does not depend on the choice of the basis $\{\bu_i\}_{i=1}^{m-n}$.

\begin{definition}
Let $\be'_1,\ldots,\be'_n$ be the standard basis of $M'_\R$, and
let $N_A$ be the the smallest positive integer
satisfying that $N_A[\ba'_{1}\cdots\ba'_{n}]^{-1}$ is an integer matrix
for any vectors $\ba'_1,\ldots,\ba'_n\in\{\ba_1,\ldots,\ba_m,d\be'_1,\ldots,d\be'_n\}$
such that $\det[\ba'_{1}\cdots\ba'_{n}]\ne0$.
\end{definition}

Now, we are ready to state our main result of this paper.
\begin{theorem}\label{thm:main}
For a Laurent series $f=a_{p}t^{p}+a_{p+1}t^{p+1}+\cdots
\in K((t))$ with $a_p\neq 0$, let
$\mathbf{in}(f):=a_{p}$ denote the initial coefficient.
If $\bw_f\not\in\mathrm{Sk}(\Fan(A))$,
then the limiting periods of $H^{n-2}_{\infty}(X_f/\Delta)$ lie in 
the $K$-algebra generated by
\[
\PG(N_A),\quad
(2\pi\ii)^{\pm1},\quad e^{\frac{\pi \ii}{N_A}},\quad
({\mathbf{in}}\,f_i)^{\frac{1}{N_A}},\,\log (\mathbf{in}\, f_i),\quad(i=1,2,\ldots,m)
\]
where $\PG(N_A):=\PG(N_A,0)$ is the $\Q$-algebra introduced in Definition \ref{defn:psi-algebra}.
\end{theorem}
We note that the condition 
$\bw_f\not\in\mathrm{Sk}(\Fan(A))$ is readily checked 
thanks to Proposition \ref{prop:skeleton}
together with \eqref{eq:weight vector}.

The special values of polygamma functions are close to those of
logarithmoc functions or
Dirichlet's $L$-values thanks to Lemma \ref{lem:main} below.
Therefore, our main theorem asserts that the limiting periods 
of a one-parameter degeneration of projective hypersurfaces
are generated by log, Gamma and Drichlet $L$-values.
\begin{lemma}\label{lem:main}
Let $0<p<q$ be coprime integers. Then
\[
\psi^{(0)}(p/q)-\psi^{(0)}(1)=-\log q-\frac{\pi}2\frac{\cos(\pi p/q)}{\sin(\pi p/q)}
+\frac12\sum_{j=1}^{q-1}\cos\left(\frac{2\pi jp}q\right)\log\left(2-2\cos\frac{2\pi j}q\right).
\]
and, for $i\geq1$, we have
\[
\psi^{(i)}(p/q)=(-1)^{i-1}\frac{i!q^{i+1}}{\varphi(q)}
\sum_{\chi}\bar\chi(p)L(i+1,\chi)
\]
where $\varphi(q):=\sharp(\Z/q\Z)^\times$, and 
$\chi$ runs over the set of primitive characters of conductor $q$.
\end{lemma}
\begin{proof}
The former equality is nothing but \cite[5.4.19]{olver2010nist}.
For $i\geq1$, the series expansion (\cite[5.15]{olver2010nist})
\[
\psi^{(i)}(z)=\sum_{n=0}^\infty\frac{(-1)^{i-1}i!}{(z+n)^{i+1}}
\]
implies \[
L(i+1,\chi)
=\sum_{n=1}^\infty\frac{\chi(n)}{n^{i+1}}
=\sum_{0<k<q}\sum_{n=0}^\infty\frac{\chi(k)}{(k+nq)^{i+1}}
=\frac{(-1)^{i-1}}{i!q^{i+1}}\sum_{0<k<q}\chi(k)\psi^{(i)}(k/q).
\]
Then, apply the orthogonal relations of characters.
\end{proof}

\subsection{Proof of Theorem \ref{thm:main}}\label{sect:Proof of the main theorem}
In this section, we prove Theorem \ref{thm:main}.
It is enough to consider the primitive part of the de Rham cohomology group of $X_f/\Delta$
because the period of the differential form of
an algebraic cycle class is $(2\pi\sqrt{-1})^r$
(\cite[p.21-23]{deligne2009hodge}).
Let
\[
H^{n-2}_\infty(X_f/\Delta)_{\mathrm{prim},K}
=(H_{\dR,\infty,K},H_{B,\infty},F^\bullet H_{\dR,\infty},W_\bullet H_{B,\infty},\iota_{\infty,K}) 
\]
denote the limiting MHS of the primitive part.
We abbreviate ``prim" on the right-hand side
because there is no fear of confusion.
Then, we identify the primitive part
$H^{n-2}_B(X_{\bt_0},\Q(-1))_{\mathrm{prim}}$ with $H^{n-1}_B(\P^{n-1}\setminus X_{\bt_0},\Q)$, and $H^{n-2}_\dR(X_f/S)_{\mathrm{prim}}$ with 
$H^{n-1}((\P^{n-1}_S\setminus X_f)/S)$.
\subsubsection{Step 1}\label{sect:step1}
First, we show that it is enough to prove Theorem \ref{thm:main}
when the polynomial
$F_f$ in \eqref{F_f} contains monomials
$x_1^d,\ldots,x_n^d$.
For this purpose, we set
\[
F_{f,s}:=F_f+sG:=F_f+s\sum_{i=1}^{n} x_i^d.
\]
This defines a smooth projective family of hypersurfaces
over $D=\{(t,s)\mid0\leq |s|\ll\varepsilon
<|t|<1\}$, which we denote by $X_{f,s}/D$.
We have the natural isomorphism $H^i_B(X_{f,s},\Q)\cong H^i_B(X_{f,0},\Q)$
that is compatible with the action of the local monodromy at $t=0$.
We put $F^{(N)}_f:=F_{f,t^N}$ and $X^{(N)}_f:=X_{f,t^N}$ for $N\geq1$.
The following symbol denotes the limiting MHS of $X_f^{(N)}$: 
\[
H^{n-2}_\infty(X^{(N)}_f/\Delta)_{\mathrm{prim},K}
=(H^{(N)}_{\dR,\infty,K},H^{(N)}_{B,\infty},F^\bullet H^{(N)}_{\dR,\infty},W_\bullet H^{(N)}_{B,\infty},\iota^{(N)}_{\infty,K}).
\]
We show that there is an isomorphism
\begin{equation}\label{eq:step1.eq1}
u:H_{\dR,\infty,K}\os{\cong}{\lra}H^{(N)}_{\dR,\infty,K}    
\end{equation}
of $K$-vector spaces that satisfies $\iota_\infty\circ u=\iota_\infty$
if $N\gg1$.
We put $U_f:=\P^{n-1}\setminus X_f$ and $U^{(N)}_f:=\P^{n-1}\setminus X^{(N)}_f$.
Then, we obtain direct sum decompositions
\[
H_\dR:=H^{n-1}_\dR(U_f/K(\!(t)\!))=\bigoplus_{\bc\in J}K(\!(t)\!)\omega_{F_f,\bc}
\]
and
\[
H^{(N)}_\dR:=H^{n-1}_\dR(U^{(N)}_f/K(\!(t)\!))=\bigoplus_{\bc\in J}K(\!(t)\!)\omega_{F_f^{(N)},\bc},
\]
where $J\subset(\Z_{\geq0})^n$ is a set of indices.
Put 
$H:=\sum_{\bc\in J}K[\![t]\!]\omega_{F_f,\bc}$ and 
$H^{(N)}:=\sum_{\bc\in J}K[\![t]\!]\omega_{F^{(N)}_f,\bc}$.
Let $\tilde u:H_\dR\to H^{(N)}_\dR$ be the isomorphism given by
$\omega_{F_f,\bc}\mapsto \omega_{F^{(N)}_f,\bc}$.
This is not compatible with the action of $\partial_t$, but satisfies that
\begin{equation}\label{eq:step1.comp.eq1}
(\tilde u\partial_t-\partial_t\tilde u)(H)\subset t^{N-l}H^{(N)}
\end{equation}
for some $l\geq0$ and large enough $N$. To see this, we expand
an element
\[
\partial_t\left(\frac{\bx^{\bc}}{F_f^r}\frac{d\bx}{\bx}\right)=-r\frac{
\partial_tF_f\,\bx^{\bc}}{F_f^{r+1}}\frac{d\bx}{\bx},
\quad
\partial_t\left(\frac{\bx^{\bc}}{(F_f+t^NG)^r}\frac{d\bx}{\bx}\right)=-r\frac{(\partial_tF_f+Nt^{N-1}G)\bx^{\bc}}{(F_f+t^NG)^{r+1}}\frac{d\bx}{\bx}
\]
in the de Rham cohomology
as a linear combination of $\omega_{F_f^{(N)},\bc}$'s.
This shows that 
$(\tilde u\partial_t-\partial_t\tilde u)(\omega_{F_f,\bc})$ lies in
$t^{N-l_{\bc}}H^{(N)}$
for some $l_{\bc}\geq0$.
The inclusion \eqref{eq:step1.comp.eq1} holds true for $l:=\max( l_{\bc})$.
We show that $\tilde u$
preserves the $V$-filtration if $N\gg1$, and hence induces
the isomorphism \eqref{eq:step1.eq1}.
Fix a large enough $k>0$ such that $t^{-k}H\supset V_{0}H_\dR$ 
and
$t^kH^{(N)}\subset V_{<-1}H^{(N)}_\dR$.
We have $(\tilde u\partial_t-\partial_t\tilde u)(t^{-k}H)\subset
t^{N-l-k}H^{(N)}\subset t^{N-l-2k}V_{-1}H^{(N)}_\dR$ by \eqref{eq:step1.comp.eq1}.
Take $N\gg1$ so that one has
$(\tilde u\partial_t-\partial_t\tilde u)(t^{-k}H)\subset V_{<-1}H^{(N)}_\dR$.
We then have the $K[\![t]\!]$-linear map
\[
\tilde u:t^{-k}H\lra H^{(N)}_{\dR}/V_{<-1}H^{(N)}_\dR
\]
that is compatible with the action of $t\partial_t$, and factors through 
$t^{-k}H/V_{-2k}H_\dR$ as $\tilde u(V_{-2k}H_\dR)\subset \tilde u(t^kH)=
t^k H^{(N)}\subset V_{<-1}H^{(N)}_\dR$.
The left hand side contains $V_0H_\dR$.
By definition of the $V$-filtration, 
one has
$\tilde u(V_\alpha H_\dR)\subset V_\alpha H^{(N)}_\dR$ for any $-1\leq \alpha< 0$
as required.

Next, we show $\iota_{\infty,K}\circ u=\iota_{\infty,K}$.
To do this, it is enough to show that for a given (large) $p>0$, there is a large 
enough $N$ such that
\[
\int_\delta\frac{\bx^{\bc}}{F_f^r}\frac{d\bx}{\bx}-
\int_\delta\frac{\bx^{\bc}}{(F^{(N)}_f)^r}\frac{d\bx}{\bx}\in t^p\C[\![t]\!]
\]
for any (piecewise smooth) cycle $\delta$ in
$\P^{n-1}$ which does not meet $X_f$ or $X^{(N)}_f$.
However, this follows from the fact that the order of 
\[
\int_\delta\frac{\bx^{\bc}}{F_f^r}\frac{d\bx}{\bx}
-\int_\delta\frac{\bx^{\bc}}{(F^{(N)}_f)^r}\frac{d\bx}{\bx}
=\int_\delta \frac{(F^{(N)}_f)^r-F_f^r}{(F_fF^{(N)}_f)^r}\bx^{\bc}\frac{d\bx}{\bx}
\]
at $t=0$ goes to infinity when $N\to\infty$.

\medskip

We have constructed the isomorphism \eqref{eq:step1.eq1}, so 
the limiting periods of $X_f/\Delta$ agree with those of $X^{(N)}_f/\Delta$.
There remains to show that the family $X^{(N)}_f/\Delta$ is adapted to the setting
of Theorem \ref{thm:main}, where $(f_1,\ldots,f_m)$ is replaced with 
$f':=(f_1,\ldots,f_m,t^N,\ldots,t^N)$.
Obviously, the conditions \eqref{eq:condition (i)} and \eqref{eq:condition (ii)} are satisfied.
Let $A':\Z^{m+n}\to \Z^n$ be an additive homomorphism such that the image of the $i$-th standard vector of $\Z^{m+n}$ is $\ba_i$ for any $i=1,\dots,m$ and $\ba_{m+j}=d\be_j\in \Z^n$ for any $j=1,\dots,n$.
Let
$\be_1^*,\ldots,\be_{m+n}^*$
be the dual basis of the standard basis $\be_1,\ldots,\be_{m+n}\in \Z^{m+n}$.
Then the vector \eqref{eq:w-vec} turns out to be
\[\bw_{f'}=\sum_{i=1}^m\ord_t(f_i)\be_i^*+N(
\be_{m+1}^*+\cdots+\be_{m+n}^*).
\]
By the assumption, 
$\bw_f$ does not belong to the skeleton 
$\mathrm{Sk}(\Fan(A))$.
Then, by Lemma \ref{lem:compare.skeleton}, this implies that
$\bw_{f'}$ does not belong to $\mathrm{Sk}(\Fan(A'))$
if $N\gg1$.
This completes the proof.

\subsubsection{Step 2}\label{sect:step2}
It is enough to prove Theorem \ref{thm:main}
when the polynomial
$F_f$ in \eqref{F_f} contains monomials
$x_1^d,\ldots,x_n^d$.
Until the end of the proof, we re-define the notation as follows.
Let 
$F=\sum_{i=1}^mz_i\bx^{\ba_i}+z_{m+1}x_1^d+\cdots+z_{m+n}x_n^d$ and
\begin{equation}\label{eq:A matrix of Fermat deformations}
A=\begin{bmatrix}
    \ba_1&\cdots&\ba_m&d\be_1&\cdots&d\be_n
\end{bmatrix}    
\end{equation}
where $\be_i={}^t(0,\ldots,1,\ldots,0)$ is the $i$-th standard column vector.
Let $X\to S\subset \Spec\C[z_i]_{1\leq i\leq m+n}$ be a projective smooth
family of hypersurfaces defined by $F$, and put $U=\P^{n-1}_S\setminus X$.
Let $f=(f_1,\ldots,f_{m+n})\in (\O_t^{an}[t^{-1}]\setminus\{0\})^{m+n}$ and
\[
F_f=\sum_{i=1}^m f_i(t)\bx^{\ba_i}+f_{m+1}(t)x_1^d+\cdots+f_{m+n}(t)x_n^d,
\]
which satisfy the conditions \eqref{eq:condition (i)}
and \eqref{eq:condition (ii)}.
Let \[
H^{n-1}_\infty(X_f/\Delta)_{\mathrm{prim},K}
=(H_{\dR,\infty,K},H_{B,\infty},F^\bullet H_{\dR,\infty},W_\bullet H_{B,\infty},\iota_{\infty,K}) 
\]
be the limiting MHS of the primitive part, where 
$H_{B,\infty}=H^{n-1}(U_{\bz},\Q)$ and $H_{\dR,\infty,K}$ is the direct sum of
the graded pieces of 
the de Rham cohomology of $U_f=\P^{n-1}_\Delta\setminus X_f$
by the $V$-filtration.
In particular, any element $\omega\in H_{\dR,\infty,K}$ is represented by a $K(\!(t)\!)$-linear combinations of $\{\omega_{\bc}\}_{\bc\in I}$
by Lemma \ref{lem:Fermat-dR-lemma}.
By assumption that $\bw_f\notin\mathrm{sk(Fan}(A))$, the
polyhedral subdivision 
$T:=S(\bw_f)$ is a regular triangulation.
Let 
\[
H^B_{n-1}(U_{\bz},\Q)\lra \Sol^T(M_A(\bc)),\quad \delta\longmapsto\int_\delta
\omega_{\bc}
\]
be the linear map which associates asymptotic expansions to homology cycles
(cf. \eqref{eq:betti}).
Let $f^*:\Sol^T(M_A(\bc))\to\C(\!(t^{1/N_A})\!)[\log t]$
be the pull-back by $t\mapsto\bz=(f_i(t))_i$.
Let 
\begin{equation}\label{eq:periods-expansion}
\int_\delta \omega_{\bc}=\sum g_{i_1\ldots i_{m+n}}(\bz)(\log z_1)^{i_1}\cdots(\log z_{m+n})^{i_{m+n}}    
\in\Sol^T(M_A(\bc))
\end{equation}
where the sum is a finite sum.
Let
\[f^*\left(\int_\delta\omega_{\bc}\right)=g_0(t)+g_1(t)\log t+\cdots+g_k(t)(\log t)^k,
\quad (\,g_i(t)\in\C(\!(t^{1/N_A})\!)\,)\]
be an expansion arond $t=0$.
To prove Theorem \ref{thm:main}, 
in view of the definition \eqref{eq:iota infty} of $\iota_\infty$, 
it is enough to prove that
every coefficients of $g_0(t)$
lie in the ring in Theorem \ref{thm:main}.
Every $g_i(t)$ are linear 
combinations of $\{f^*g_{i_1\ldots i_{m+n}}(\bz)\}_{i_1,\ldots,i_{m+n}}$ over a ring 
$K[\![t]\!][\log(\mathbf{in}f_j)]_{1\leq j\leq m+n}$.
For any $\alpha\in\frac{1}{N_A}\Z$,
the coefficient of $f^*g_{i_1\ldots i_{m+n}}(\bz)$ of $t^\alpha$ 
is a (finite) linear 
combination of coefficients of the series $g_{i_1\ldots i_{m+n}}(\bz)$ 
over a field $K[(\mathbf{in}f_j)^{1/N_A}]_{1\leq j\leq m+n}$
by Lemma \ref{prop:finite truncation}.
Therefore, the proof of Theorem \ref{thm:main} is reduced to
showing the following assertion. 
\begin{equation}\label{eq:step2-quot}
\text{
Any coefficients of 
$g_{i_1\ldots i_{m+n}}(\bz)$ 
lie in a ring $R_{A}:=\Image[\PG(N_A)\ot \Q(e^{\pi\ii/N_A})\to\C]$.
}    
\end{equation}
To prove this, we trace the proof of 
Theorem \ref{thm:arbitrary triangulation}.
When $T=T(\Fer)$, the assertion \eqref{eq:step2-quot} is immediate from Lemmas 
\ref{lem:key-lem} and \ref{lem:image of the integration map}.
Let $T$ be arbitrary.
There is a sequence of regular triangulations $T_1=T(\Fer),T_2,\ldots,T_\ell=T$ such that
$T_i$ and $T_{i+1}$ share a facet for each $i$.
Let $\bc(\varepsilon)=\bc+\varepsilon\bc'$ with $\bc'\in\Q^n$ a generic vector.
Recall the connection formula (Theorem \ref{thm:connection formula}), which allows us
to describe the isomorphism
$\rho_i:\Sol^{T_i}(M_A(\bc(\varepsilon)))\os{\sim}{\to}\Sol^{T_{i+1}}(M_A(\bc(\varepsilon)))$ explicitly.
By Lemma \ref{prop:connection formula},
the isomorphism preserves the submodules of the SST-limits over $R_A$,
\[
\rho_i\left(\lim_{\varepsilon\to0}\Sol^{T_i}(M_A(\bc(\varepsilon)))_{R_A}\right)= \lim_{\varepsilon\to0}\Sol^{T_{i+1}}(M_A(\bc(\varepsilon)))_{R_A}.
\]
Since the period integral $\int_\delta\omega_{\bc}$ belongs to the SST-limit
$\lim_{\varepsilon\to0}\Sol^{T(\Fer)}(M_A(\bc(\varepsilon)))_{R_A}$ over $R_A$ by
Lemmas 
\ref{lem:key-lem} and \ref{lem:image of the integration map}, we have
\[
\int_\delta\omega_{\bc}\in \lim_{\varepsilon\to0}\Sol^T(M_A(\bc(\varepsilon)))_{R_A}.
\]
Now the assertion \eqref{eq:step2-quot} follows from Theorem \ref{thm:coefficients}. This completes the proof.

\section{Perturbation of the generalized Dwork family}\label{sect:Dwork family}
In this section, we consider limiting periods of a family $X_f/\Delta$ when the defining equation $F$ is given by \eqref{eq:Fermat deformation} and $f$ satisfies a condition described below.
Roughly speaking, we describe a condition that  $X_f/\Delta$ is close to those of the generalized Dwork family.
A detailed computation for this class of degeneration yields a refinement of the connection formula as in Theorem \ref{thm:analytic continuation of p.Dwork}, which provides a concrete class of examples of Theorem \ref{thm:main}.

\subsection{Limiting periods of the generalized Dwork family}\label{subsec:Dwork family}
In the literature of periods, a well-studied class of
hypersurfaces is the {\it generalized Dwork family} 
by Katz \cite{katz2010another}.
It is a special class of Fermat deformation \eqref{eq:Fermat deformation} where $m=1$:
\begin{equation}\label{eq:generalized.Dwork.eq}
F:=z_1\bx^{\ba}+z_2x_1^{d}+\cdots+z_{n+1}x_n^{d},
\end{equation}
where $\ba=(a_1,\dots,a_n)$.
We write $D(z_1,\dots,z_{n+1};\ba)$ for the corresponding family $X/S$ of projective hypersurfaces.
We work over the ground field $K=\Q$.
Put\[
A:=\begin{bmatrix} a_1&d\\
\vdots&&\ddots\\
a_n&&&d
\end{bmatrix},\quad \bu:=\begin{bmatrix}
    d\\ -a_1\\ \vdots \\ -a_n
\end{bmatrix}.
\]
We note that there are exactly two regular triangulations: one is $T_1=T(\Fer)$ whose maximal simplex is exactly $\{ 2,\dots,n+1\}$; the other one is $T_2$ whose maximal simplices are $\{ \{1,2,\dots, \widehat{(j+1)},\dots,n+1\}\mid j=1,\dots,n,\ \ a_{j}\neq 0\}$.
We first focus on $T_1$.
Let $\bc\in\Z_{> 0}^n$ satisfy 
$|\bc|=rd$.
The vector defined in \eqref{eq: gamma.bc.bp} reads
\[
\bg^{\bc}:=\begin{bmatrix}0\\ -\frac1d c_1\\ \vdots \\ -\frac1d c_n\end{bmatrix},\quad
\bg^{\bc}_p=\bg^{\bc}+\frac{p}{d}\bu,\quad(p\in\Z).
\]
Let $P=P_{\bc}\subset\{ p\in\Z|0\leq p<d \}$ be a finite set of integers defined as in \eqref{eq:P_c}.
In what follows, we assume $\gcd(d,a_1,\ldots,a_n)=1$.
Then, $L_A$ is generated by $\bu$, so that one has
\begin{equation}\label{eq:phi_p}
\varphi(\bg^{\bc}_p;\bz)=\sum_{i=0}^\infty\frac{1}{\Gamma(1+\bg^{\bc}_p+i\bu)}\bz^{\bg^{\bc}_p+i\bu},
\quad(0\leq p<d).
\end{equation}
For $\bg_1,\bg_2\in \C^n$, 
we choose $\bg_1',\bg_2'\in\C^{n'}$ so that an identity
\[
\prod_{i=1}^n\frac{x-\gamma_{1i}}{x-\gamma_{2i}}
=\prod_{i=1}^{n'}\frac{x-\gamma'_{1i}}{x-\gamma'_{2i}}
\]
holds true in the field of rational functions $\C(x)$.
Here, the right-hand side is an irreducible fraction.
We then define
\[
\text{\rm Cancell}
\left({\bg_1\atop \bg_2}\right):=
\left({\bg'_1\atop \bg_2'}\right)
\]
and say that $\ba',\bb'$ is the {\it cancellation} of $\ba,\bb$.
Obviously, the cancellation is unique up to order.
We put
\begin{align*}
&\widetilde\balpha=\left(\frac{\ell}{a_k}+\frac{c_k}{da_k}\right);\text{$k,\ell$ such that }a_k>0\text{ and }0\leq 
\ell<a_k,\\
&\widetilde\bbeta=\left(1,\frac1d,\frac2d,\ldots,\frac{d-1}d\right),
\end{align*}
and put
\begin{equation}\label{eq:cancell-2}
\left({\balpha\atop \bbeta}\right)=
\text{Cancell}\left({\widetilde\balpha\atop
\widetilde\bbeta}\right).
\end{equation}
RWe reorder $\balpha$ so that
\begin{equation}\label{eq:cancell-3}
\balpha=(\overbrace{\alpha_1,\ldots,\alpha_1}^{s_1},\overbrace{\alpha_2,\ldots,\alpha_2}^{s_2},\ldots,
\overbrace{\alpha_N,\ldots,\alpha_N}^{s_N}),\quad
\alpha_1<\cdots<\alpha_N.
\end{equation}
Note that 
\[
0<\frac{\ell}{a_k}+\frac{c_k}{da_k}=\frac{1}{a_k}(\ell+\frac{c_k}{d})<\frac{1}{a_k}(\ell+1)\leq1.
\]
In particular, one has $\alpha_i-\beta_j\not\in\Z$ for any $i,j$.
The following lemma immediately follows from the definition.

\begin{lemma}\label{lem:lemma P}
$\bbeta=\left(1-\frac{p}d; p\in P\right)$.
In particular, $\sharp P=s_1+s_2+\cdots+s_N$.
\end{lemma}

For $\bg,\bg'\in \C^{m}$ such that $\gamma_i\not\in\Z_{\leq0}$ for all $i$,
 we define a series $\cF(\bg,\bg';z)$ as follows,
 \[
 \cF(\bg,\bg';z):=\sum_{i=0}^\infty\prod_{k=1}^m
 \frac{\Gamma(\gamma_k+i)}{\Gamma(\gamma'_k+i)}z^i.
 \]
 The following proposition is an easy consequence of the multiplication formula of the Gamma function \cite[5.5.6]{olver2010nist}.
 We put
\[
z:=\left((-d)^{-d}\prod_{a_k>0}a_k^{a_k}\right)\bz^{\bu}.
\]
\begin{lemma}\label{thm:ex1}
For $p\in P$ one has the following identity:
\[
\varphi(\bg^{\bc}_p;\bz)=
\frac{\bz^{\bg_p^{\bc}}}{\Gamma(1+\bg_p^{\bc})}\frac{\Gamma(\frac{p}{d}+\bbeta)}
{\Gamma(\frac{p}{d}+\balpha)}
\cF(\frac{p}{d}+\balpha;\frac{p}{d}+\bbeta;z).
\]
\end{lemma}
By Lemmas \ref{thm:ex1}, we get the Picard-Fuchs equation
of the generalized Dwork family. The following theorem is due to Negishi, proved 
in a different way from ours.
\begin{theorem}[{Negishi, \cite{negishi2022picard}}]
$\bt^{-\bg^{\bc}}\omega_{\bc}$ is annihilated by the hypergeometric operator
\[
P=\prod_{\beta_i\in\bbeta}(D_z+\beta_i-1)-z
\prod_{\alpha_i\in\balpha}(D_z+\alpha),\quad\left(D_z:=z\frac{d}{dz}\right).
\]
\end{theorem}

\begin{theorem}\label{thm:MB}
    Let $d,n$ be positive integers and let $\balpha=(\alpha_1,\dots,\alpha_n)\in\C^n$ and $\ba=(a_1,\dots,a_n)\in\Z_{\geq 0}^n$.
    We assume that $|\ba|=d$, $0\leq p<d$ is an integer, and that ${\rm  Re}(\alpha_i-a_i)>0$ for any $i=1,\dots,n$ such that $a_i\neq 0$.
    We set    \begin{equation}\label{eq:lacunary series}
        \varphi_p(\zeta;\ba,\balpha):=\sum_{k=0}^\infty \frac{\Gamma(\balpha+k \ba)}{\Gamma(1+p+dk)}\zeta^{dk}.
    \end{equation}
    Then, for any $0<c<1$, the integral
    \begin{align}
         &I_p(\zeta;\ba,\balpha)\nonumber\\
        =&\frac{1}{2\pi\ii}\frac{2^{d-1}}{d}\int_{c+d-\ii\infty}^{c+d+\ii\infty}\frac{\Gamma(s)\displaystyle\prod_{i=1}^{d-1}\sin\left(\frac{\pi(s+i)}{d}\right)}{(1-s)_p}\Gamma\left(\balpha-\frac{\ba}{d}s\right)(e^{-\frac{\pi}{d}\ii}\zeta)^{-s}ds\label{eq:integral}\\
        =&\frac{(-1)^p}{2\pi\ii}{2^{d-1}}\int_{c-\ii\infty}^{c+\ii\infty}{\Gamma(ds)\displaystyle\prod_{i=1}^{d-1}\sin\left({\pi(s+\frac{p+i}{d})}\right)}\Gamma\left(\balpha-\frac{p}{d}{\ba}-{\ba}s\right)(e^{-\frac{\pi}{d}\ii}\zeta)^{-d{s}-p}ds\label{eq:integral2}
    \end{align}
    is convergent when
    \begin{equation}\label{eq:condition}
        |{\rm arg}(e^{-\frac{\pi}{d}\ii}\zeta)|<\frac{\pi}{d}.
    \end{equation}
    Here, we set $(1-s)_p:=(1-s)(2-s)\cdots (p-s)$.
    Moreover, the integral \eqref{eq:integral} coincides with the series \eqref{eq:lacunary series} if $|\zeta|<\!\!<1$ and it coincides with 
    \begin{equation}\label{eq:residues}
        (-1)^{p+1}{2^{d-1}}\sum_{s_0\in S}\underset{s=s_0}{\rm Res}\left\{{\Gamma(ds)\displaystyle\prod_{i=1}^{d-1}\sin\left({\pi(s+\frac{i}{d})}\right)}\Gamma\left(\balpha-\frac{p}{d}{\ba}-{\ba}s\right)(e^{-\frac{\pi}{d}\ii}\zeta)^{-d{s}-p}\right\}
    \end{equation}
    if $|\zeta|>\!\!>1$. 
    Here $S$ is the set of poles of the integrand on the right half plane 
    \begin{equation}\label{eq:right half plane}
        \{s\in\C\mid {\rm Re}(s)>0\}.
    \end{equation}
\end{theorem}

\begin{proof}
    The integral \eqref{eq:integral} is equal to the one \eqref{eq:integral2} by a change of integration variable.
    Let us first confirm that the integral \eqref{eq:integral} is convergent.
    By Stirling's formula, for any $0<\varepsilon<\!\!<1$, there exists a constant $C_\varepsilon>0$ so that the Gamma function decays as
    \begin{equation}
        |\Gamma(c+\ii t)|\leq C_\varepsilon e^{-(\frac{\pi}{2}-\varepsilon)|t|}
    \end{equation}
    as $t\to\pm\infty$.
    Thus, for any $0<\varepsilon<\!\!<1$, there exists a constant $C_\varepsilon>0$ so that the product $\Gamma(s)\Gamma\left(\balpha-\frac{\ba}{d}s\right)$ decays as
    \begin{equation}
        |\Gamma(c+\ii t)\Gamma\left(\balpha-\frac{\ba}{d}(c+\ii t)\right)| \leq C_\varepsilon e^{-({\pi}-\varepsilon)|t|}
    \end{equation}
    as $t\to\pm\infty$.
    On the other hand, the function $\prod_{i=1}^{d-1}\sin\left(\frac{\pi(s+i)}{d}\right)$ is bounded as 
    \begin{equation}
        \left|\prod_{i=1}^{d-1}\sin\left(\frac{\pi(c+\ii t+i)}{d}\right)\right|\leq Ce^{\frac{(d-1)\pi}{d}|t|}.
    \end{equation}
    Therefore, the convergence of the integral \eqref{eq:integral} is true under the condition \eqref{eq:condition}.
    The rest of the proof is residue calculus.
    The integrand of \eqref{eq:integral} is a meromorphic function on the domain
    \begin{equation}\label{eq:half plane}
        \{s\in\C\mid {\rm Re}(s)<c+d\}.
    \end{equation}
    The poles in this domain are $-d\Z_{\geq 0}$.
    The residue of $\Gamma(s)$ at the pole $s=-dk$ is $\frac{(-1)^{dk}}{(dk)!}$.
    We note that the combinatorial identity 
    \begin{equation}
        \prod_{i=1}^{d-1}\sin\left(\frac{\pi i}{d}\right)=\frac{d}{2^{d-1}}
    \end{equation}
    holds true.
    In fact, $\prod_{i=1}^{d-1}\sin\left(\frac{\pi i}{d}\right)=\frac{1}{2^{d-1}}\prod_{i=1}^{d-1}(1-e^{\frac{2\pi i}{d}\ii })$ and 
    \begin{equation}
        \prod_{i=1}^{d-1}(1-e^{\frac{2\pi i}{d}\ii })=\lim_{x\to 1}\frac{x^d-1}{x-1}=d.
    \end{equation}
    Thus, the series \eqref{eq:lacunary series} is the sum of residues of the integrand of \eqref{eq:integral} on the domain \eqref{eq:half plane}.
    Similarly, \eqref{eq:residues} is the sum of residues of the integrand of of \eqref{eq:integral2} on the domain \eqref{eq:right half plane}.
\end{proof}

\begin{remark}\label{rem:integration contour}
    When the assumption ${\rm Re}(\alpha_i-a_i)>0$ is violated but some other generic condition is fulfilled (e.g., ${\rm Im}(\alpha_i)\neq{\rm Im}(\alpha_j)$ for any $i\neq j$), one can continuously deform the integration contour of \eqref{eq:integral2} so that it separates the poles of $\Gamma(ds)$ from those of $\Gamma\left( \balpha-\frac{p}{d}\ba-\ba s\right)$.
\end{remark}

We set $\tilde\bz:=(z_2,\dots,z_{n+1})$.
A direct computation shows the following identity: 
\begin{align}
\varphi(\gamma^{\bc}_{\bp};\bz)&=(-1)^{p}\frac{\sin\pi(\frac{\bc}{d}+\frac{p}{d}\ba)}{\pi^n}\sum_{i\geq 0}\frac{\Gamma\left(\frac{\bc}{d}+(\frac{p}{d}+i)\ba\right)}{\Gamma(1+p+id)}(-z_1)^{p+id}\tilde{\bz}^{-\frac{\bc}{d}-(\frac{p}{d}+i)\ba}\nonumber\\
    &=z_1^p\tilde{\bz}^{-\frac{\bc}{d}-\frac{p}{d}\ba}\frac{\sin\pi(\frac{\bc}{d}+\frac{p}{d}\ba)}{\pi^n}\varphi_p(-z_1\tilde{\bz}^{-\frac{\ba}{d}};\ba,\frac{\bc}{d}+\frac{p}{d}\ba).\label{eq:computation}
\end{align}
The set of poles $S$ in \eqref{eq:residues} is given by $S=\{\alpha_i+j\mid i=1,\dots,N,\ j=0,1,\dots\}$.

\noindent\textbf{Example}.
Let $(n,d)=(6,7)$, $\ba=(2,1,\ldots,1)$ and $F=z_1\bx^{\ba}+z_2x_1^7+\cdots+z_7x_6^7$.
The index set $I$ defined by \eqref{eq:I} has $6,666$ elements and it is subdivided into $2,401$ equivalence classes \eqref{eq:definition:Ic}.
The only possible cardinalities of $I_{\bc}$ are $2,3,4$, or $6$, each of which has $1,080$, $780$, $540$ and $1$ kinds respectively.
Let us choose $\bc=(1,2,1,1,1,1)$.
Then, $P_{\bc}=\{0,1,2,4\}$ and 
\[
\balpha=(\frac17,\frac17,\frac17,\frac{1}{14}),\quad\bbeta=(1,\frac37,\frac57,\frac67).
\]
We set $\zeta=-z_1\tilde{\bz}^{-\frac{\ba}{d}}$.
The series $\varphi_p(\zeta;\ba,\frac{\bc}{d}+\frac{p}{d}\ba)$ is well-defined for $|\zeta|<\!\!<1$.
We choose a path of analytic continuation of $\varphi_p(\zeta;\ba,\frac{\bc}{d})$ as a straight line emanating from the origin towards infinity in a sector $\frac{\pi}{d}<{\rm arg}(\zeta)<\frac{2\pi}{d}$.
The result of analytic continuation $\varphi(\bg^{\bc}_{\bp};\bz)$ is
\begin{equation}
    z_1^p\tilde{\bz}^{-\frac{\bc}{d}-\frac{p}{d}\ba}\left[ C_{\frac{1}{14}0}^{(p)}\zeta^{-\frac{1}{2}}f_p(\zeta^{-1})+\sum_{i=0}^2C_{\frac{1}{7}i}^{(p)}f_{pi}(\zeta^{-1})\zeta^{-1}(\log\zeta)^i\right],
\end{equation}
where $f_p(z),f_{pi}(z)$ are power series with $f(0)=f_i(0)=1$.
For example, we obtain that
\begin{equation}
    C_{\frac{1}{14}0}^{(0)}=\frac{32 \sin \left(\frac{\pi }{14}\right) \sin
   ^5\left(\frac{\pi }{7}\right) \sin ^2\left(\frac{3 \pi
   }{14}\right) \sin
   \left(\frac{2 \pi }{7}\right) \sin ^2\left(\frac{5\pi }{14}\right)  \Gamma
   \left(\frac{1}{14}\right)^4 \Gamma
   \left(\frac{3}{14}\right)}{\pi ^{11/2}},\ \ \text{and}
\end{equation}

\begin{align}
    C_{\frac{1}{7}0}^{(0)}=&-\frac{49 \sin ^4\left(\frac{\pi }{7}\right) \sin \left(\frac{2 \pi }{7}\right) \Gamma \left(\frac{1}{7}\right) \Gamma
   \left(\frac{6}{7}\right)}{4 \pi ^5}\nonumber\\
   &\left(2 \left(3 \gamma +3\psi ^{(0)}\left(\frac{6}{7}\right)+14\right)^2-2
   \left(\psi ^{(1)}\left(\frac{6}{7}\right)+4 \psi ^{(1)}\left(\frac{8}{7}\right)\right)\right.\nonumber\\
   &\left.+\pi ^2 \left(17+8 \csc ^2\left(\frac{\pi
   }{7}\right)-4 \sec ^2\left(\frac{\pi }{14}\right)-4 \sec ^2\left(\frac{3 \pi }{14}\right)\right)\right).\label{eq:C_0}
\end{align}
Here, $\gamma=-\psi^{(0)}(1)$ is the Euler's constant.
Note that the numbers $C_{\frac{1}{14}0}$ and $C_{\frac{1}{7}0}$ are non-zero and $\Gamma(\frac{1}{2})=\sqrt{\pi}$.
Thus, these numbers belong to the ring $\PG(14)$.

\subsection{Perturbation of the generalized Dwork family}\label{subsec:p.Dwork}
In this section, we consider the polynomial $F$ defined by \eqref{eq:Fermat deformation}
We fix $f_i(t)\in K(\!(t)\!)$ $(i=1,\dots,m+n)$ so that the vector $\bw_f$ defined as in 
\eqref{eq:w-vec} does not lie on the skeleton ${\rm Sk}({\rm Fan}(A))$.
In this section, we are interested in computing the limiting periods of the family $X_f/\Delta$ when it is "close" to the limiting period of the generalized Dwork family studied in \S\ref{subsec:Dwork family}.
We begin with formulating this notion.

\begin{definition}
    A weight vector $\bw\in\R^{m+n}$ is a perturbation of the generalized Dwork family $D(z_1,z_{m+1},\dots,z_{m+n};\ba_1)$ if 
    \begin{equation}\label{eq:5.1}
        w_1d<\sum_{k=1}^nw_{m+k}a_{k1}
    \end{equation}
    and
    \begin{equation}\label{eq:5.2}
        (a_{j\ell}w_1-a_{j1}w_\ell)d< \sum_{k=1}^nw_{m+k}(a_{j\ell}a_{k1}-a_{j1}a_{k\ell})
    \end{equation}
    for any $\ell=2,\dots,m$ and any $j=1,\dots,n$ such that $a_{j1}\neq 0$.
    A family $X_f/\Delta$ is said to be a perturbation of the generalized Dwork family $D(z_1,z_{m+1},\dots,z_{m+n};\ba_1)$ if so is the weight vector $\bw_f$.
\end{definition}

\begin{proposition}
    A weight vector $\bw\in\R^{m+n}$ is a perturbation of the generalized Dwork family $D(z_1,z_{m+1},\dots,z_{m+n};\ba_1)$ if and only if the regular polyhedral subdivision $S(\bw)$ is a regular triangulation whose maximal simplices are $T(\ba_1):=\{ \{1,(m+1),\dots, \widehat{(m+j)},\dots,(m+n)\mid j=1,\dots,n,\ \ a_{j1}\neq 0\}$.
    Here $\widehat{(m+j)}$ means that the element $(m+j)$ is omitted.
\end{proposition}

\begin{proof}
Firstly, it is easily seen that $T(\ba_1)$ is a regular triangulation.
In fact, one can take a weight vector $\bw$ so that $w_1\ll w_{m+1},\dots,w_{m+n}\ll w_2,\dots,w_m$ to obtain $T(\ba_1)=S(\bw)$.
    Let us consider an $(m+n)\times m$ matrix
    \begin{equation}\label{eq:B matrix}
    B=
    \begin{pmatrix}
        d&&\\
         &\ddots&\\
         &&d\\
         -\ba_1&\cdots&-\ba_m
    \end{pmatrix}
    .
    \end{equation}
    The column vectors of $B$ forms a basis of $L_{A,\R}$ where $A$ is an $n\times (m+n)$ matrix as in \eqref{eq:A matrix of Fermat deformations}.
    It gives rise to an isomorphism of vector spaces $L_{A,\R}\simeq \R^m$, which also induces another isomorphism $L_{A,\R}^*\simeq \R^m$, where the latter $\R^m$ is regarded as the dual vector space of the former $\R^m$ via the dot product.
    Let $\bb_1,\dots,\bb_{m+n}$ be row vectors of $B$.
    By \eqref{eq:bar C} and \eqref{eq:bar C sigma gale}, the cone $\bar{C}_{T(\ba_1)}$ is given by
    \begin{equation}\label{eq:5.3}
\bar{C}_{T(\ba_1)}=\bigcap_{j}\big(\R_{>0}\bb_2+\cdots+\R_{>0}\bb_m+\R_{>0}\bb_{m+j}\big)
    \end{equation}
    where $j=1,\dots,n$ runs over indices such that $a_{j1}\neq 0$.
    Under the identification $L_{A,\R}^*\simeq \R^m$, a vector $\bv=(v_1,\dots,v_m)\in \R^m$ belongs to a simplicial cone $\big(\R_{>0}\bb_2+\cdots+\R_{>0}\bb_m+\R_{>0}\bb_{m+j}\big)$ if and only if
    \begin{equation}\label{eq:5.4}
        v_1\leq 0\quad\text{and}\quad a_{j\ell}v_1\leq a_{j1}v_\ell,\ \ \ell=2,\dots,m.
    \end{equation}
    Under the identification $L_{A,\R}^*\simeq \R^m$, the quotient map $\pi_A:\R^{m+n}\to L_{A,\R}^*$ is identified with a map
    \begin{equation}\label{eq:5.5}
        \R^{m+n}\ni\bw=(w_1,\dots,w_{m+n})\mapsto\left( w_1d-\sum_{k=1}^nw_{m+k}a_{k1},\dots,w_md-\sum_{k=1}^nw_{m+k}a_{km}\right)\in\R^m.
    \end{equation}
    Combining \eqref{eq:5.3}, \eqref{eq:5.4} and \eqref{eq:5.5}, we obtain that a vector $\bw$ belongs to $C_{T(\ba_1)}=\pi_A^{-1}(\bar{C}_{T(\ba_1)})$ if and only if \eqref{eq:5.1} and \eqref{eq:5.2} are both satisfied.
\end{proof}

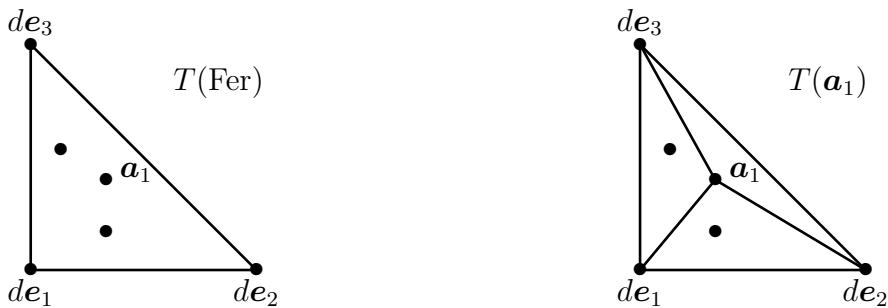
\begin{figure}[h]
  \begin{minipage}[b]{0.48\columnwidth}
\centering
\begin{tikzpicture}
\node at(2.5,2.5){$T(\Fer)$};
 \draw[line width=1pt]
(0,0)--(3,0)--(0,3)--(0,0);
\node at(0,0-0.3){$d\be_1$};
\node at(3,0-0.3){$d\be_2$};
\node at(0,3+0.3){$d\be_3$};
\node at(1,1.2){$\bullet$};
\node at(1.4,1.3){$\ba_1$};
\node at(0.4,1.6){$\bullet$};
\node at(1,0.5){$\bullet$};
\node at(0,0){$\bullet$};
\node at(3,0){$\bullet$};
\node at(0,3){$\bullet$};
\end{tikzpicture}
\end{minipage}
\begin{minipage}[b]{0.48\columnwidth}
\centering
\begin{tikzpicture}
\node at(2.5,2.5){$T(\ba_1)$};
 \draw[line width=1pt]
(0,0)--(3,0)--(0,3)--(0,0);
 \draw[line width=1pt]
(0,3)--(1,1.2)--(0,0);
 \draw[line width=1pt]
(3,0)--(1,1.2);
\node at(0,0-0.3){$d\be_1$};
\node at(3,0-0.3){$d\be_2$};
\node at(0,3+0.3){$d\be_3$};
\node at(1.4,1.3){$\ba_1$};
\node at(0.4,1.6){$\bullet$};
\node at(1,1.2){$\bullet$};
\node at(1,0.5){$\bullet$};
\node at(0,0){$\bullet$};
\node at(3,0){$\bullet$};
\node at(0,3){$\bullet$};
\end{tikzpicture}
\end{minipage}
\caption{Regular triangulations $T(\Fer)$ and $T(\ba_1)$.}
\end{figure}

\begin{remark}
The cones $\bar{C}_{T(\Fer)}$ and $\bar{C}_{T(\ba_1)}$ share a facet in the secondary fan.
In fact, it is not hard to see that any maximal cone $\bar{C}_T$ that shares a facet with $\bar{C}_{T(\Fer)}$ is of the form $\bar{C}_T=\bar{C}_{T(\ba_i)}$ for some $i=1,\dots,m$.
To see this, we only need to recall that $\bar{C}_T$ that shares a facet with $\bar{C}_{T(\Fer)}$ if and only if $T$ is a modification of $T(\Fer)$ with respect to a circuit which contains a non-empty subset of $\{d\be_1,\dots,d\be_n\}$ (\cite[Chapter 7, section 2]{GKZbook}).
On the other hand, a circuit which contains a non-empty subset of $\{d\be_1,\dots,d\be_n\}$ is of the form $\{ d\be_{j_1},\dots,d\be_{j_r},\ba_i\}$.
The result of the modification of $T(\Fer)$ along this circuit is nothing but $T(\ba_1)$.
\end{remark}

In the following, we assume that $X_f/\Delta$ is a perturbation of the generalized Dwork family $D(z_1,z_{m+1},\dots,z_{m+n};\ba_1)$.
We also assume that the column vectors of the matrix \eqref{eq:B matrix} forms a $\Z$-basis of $L_A$. 
If $(u_1,\dots,u_{m+n})\in L_A$, it follows from the definition of  $A$ \eqref{eq:A matrix of Fermat deformations} that $a_{i1}u_1+\cdots+a_{im}u_m\equiv0\mod d$ for $i=1,\dots,n$.
Therefore, what we assume is that these equations have only trivial solutions $u_{1},\dots,u_m$ modulo $d$.
More explicitly, we assume the following condition:
\begin{equation}\label{eq:a condition}
m\leq n-1\quad\text{and}\quad \gcd(d,\{\Delta_{i_1,\dots,i_m}\}_{1\leq i_1<\cdots<i_m\leq n})=1,
\end{equation}
where $\Delta_{i_1,\dots,i_m}$ denotes $({i_1,\dots,i_m})$-minor of an $n\times m$ matrix whose $(i,j)$-entry is $a_{ij}$.
Note that $|\ba_i|=d$ implies that $\bar{\ba}_1,\dots,\bar{\ba}_m$ are linearly dependent on $(\Z/d\Z)^n$.

To compute the limiting periods of $X_f/\Delta$, it is enough to compute a $\Q(\zeta_d)$-basis of the image of the map $\Q(\zeta_d)\otimes_{\Q}H_{n-1}^B(U_{\bz},\Q)\to\Sol^{T(\ba_1)}(H^{n-1}_\dR(U/S))$ defined by \eqref{eq:betti}.
For a fixed $\bc\in\{1,\dots,d-1\}^n$, the image of the map \eqref{eq:integral of omega_c} has a $\Q(\zeta_d)$-basis
$\{(2\pi\ii)^{n}\varphi(\bg^{\bc}_{\bp};\bz)\}_{\bp\in P_{\bc}}$ by Theorem \ref{thm:image of the integral of omega_c}.
For any $\bp\in P_{\bc}$, we shall apply Theorem \ref{thm:MB} to the series $\varphi(\gamma^{\bc}_{\bp};\bz)$.
Note that $\varphi(\gamma^{\bc}_{\bp};\bz)$ under the assumption \eqref{eq:a condition} takes the following form:
\begin{equation}
\varphi(\gamma^{\bc}_{\bp};\bz)
=
\sum_{(i_1,\dots,i_m)\in\mathbb{Z}_{\geq 0}^m}\frac{\prod_{j=1}^mz_j^{p_j+i_jd}\tilde{{\bz}}^{-\frac{\bc}{d}-\frac{1}{d}\sum_{j=1}^m(p_j+i_jd)\ba_j}}{\prod_{j=1}^m\Gamma(1+p_j+i_jd)\Gamma\left(1-\frac{\bc}{d}-\frac{1}{d}\sum_{j=1}^m(p_j+i_jd)\ba_j\right)},
\end{equation}
where we set $\tilde{\bz}=(z_{m+1},\dots,z_{m+n})$.
For $\bc\in\C^n$, $\bp=(p_1,\dots,p_m)\in \Z^{m}$ and an index $1\leq i\leq n$ such that $a_{i1}\neq 0$, let $\bg^{\bc}_{\bp}(\ba_1,i)\in\C^{m+n}$ denote the unique solution $\bg\in\C^{m+n}$ to the equation 
\begin{equation}
    A\bg=-\bc,\quad\gamma_{j}=p_j\ (2\leq j\leq m)\quad\text{and}\quad\gamma_{m+i}=p_1.
\end{equation} 
More explicitly, the first coordinate  of $\bg^{\bc}_{\bp}(\ba_1,i)$ is
\begin{equation}
    -\frac{1}{a_{i1}}\left(c_i+p_1d+\sum_{k=2}^mp_k a_{ik}\right)
\end{equation}
and the $(m+j)$-th coordinate for $1\leq j\leq n,j\neq i$ is
\begin{equation}
    \frac{1}{d}\left(\frac{a_{j1}}{a_{i1}}\left(c_i+p_1d+\sum_{k=2}^mp_k a_{ik}\right)-c_j-\sum_{k=2}^m p_k a_{jk}\right).
\end{equation}
Finally, for $\bq\in\Z^m$, we introduce a notation $C(\ba_1,i,\bp,\bq)$ as follows:
\begin{equation}\label{eq:Coeff}
\begin{split}
    C(\ba_1,i,\bp,\bq,\bc):=&\frac{(-1)^{p_2+\cdots+p_m}2^{d-1}}{a_{i1}}\sin\frac{\pi}{d}\left(\bc+{p_1}\ba_1+\cdots+{p_m}\ba_m\right)\times\\
    &\frac{(-1)^{|\bq|}}{\sin\pi\bg^{\bc}_{\bq}(\ba_1,i)_1\prod_{j=1,j\neq i}^n\sin\pi\bg^{\bc}_{\bq}(\ba_1,i)_{m+j}}\times\\
    &\prod_{j=1}^{d-1}\sin\frac{\pi}{d}\left(\bg^{\bc}_{\bq}(\ba_1,i)_1+p_1+j\right)\exp\left\{\frac{(d-1)\pi\ii}{d}\bg^{\bc}_{\bq}(\ba_1,i)_1\right\}.
\end{split}
\end{equation}
We note that $C(\ba_1,i,\bp,\bq,\bc)=C(\ba_1,i,\bp,\bq',\bc)$ if $\bu=\bg^{\bc}_{\bq'}(\ba_1,i)-\bg^{\bc}_{\bq}(\ba_1,i)\in L_A$.
In fact, the second factor of \eqref{eq:Coeff} is invariant by the replacement $\bg^{\bc}_{\bq}(\ba_1,i)\mapsto \bg^{\bc}_{\bq}(\ba_1,i)+\bu$ in view of the identity $|\bu|=0$.
The third factor is also invariant because $u_1$ is always a multiple of $d$.
Let $\bar{\ba}_1$ denote the equivalence class of $\ba_1$ in $M:=\Z/d\Z\times\cdots\times\Z\times\cdots\times\Z/d\Z$, where $\Z$ is the $i$-th factor.
We set $M_1:=M/\Z\bar{\ba}_1$.
We consider a map $h:\Z^m\ni(i_1,\dots,i_m)\mapsto i_1d\be_{i}+(p_2+i_2d)\ba_2+\cdots+(p_m+i_nd)\ba_m\in M_1$ and let $\{\bq_r(\bp,i)\}_{r=1}^{r_i}$ be a complete system of representatives of $\mathbb{Z}^m/{\rm Ker}(h)$.

\begin{theorem}\label{thm:analytic continuation of p.Dwork}
    Assume that $\bc\in\C^{n}$ is very generic. Then, there is a path of analytic continuation $\gamma$ from $U_{T(\Fer)}$ to $U_{T(\ba_1)}$  such that $\varphi(\bg^{\bc}_{\bp};\bz)$ is analytically continued to
    \begin{equation}\label{eq:basis of perturabation of gDwork}
        \sum_{i=1;a_{i1}\neq 0}^n\sum_{r=1}^{r_i}C(\ba_1,i,\bp,\bq_r(\bp,i),\bc)\varphi(\bg^{\bc}_{\bq_r(\bp,i)}(\ba_1,i);\bz).
    \end{equation}
    When $\bc$ is not very generic, the right-hand side is still a valid formula as an SST-limit.
\end{theorem}

\begin{proof}
    
By the same computation as \eqref{eq:computation}, the following identities hold true:

\begin{equation}\label{eq:varphi convenient}
\begin{split}
    \varphi(\gamma^{\bc}_{\bp};\bz)=&(-1)^{|\bp|}\frac{\sin\pi(\frac{\bc}{d}+\frac{p_1}{d}\ba_1+\cdots+\frac{p_m}{d}\ba_m)}{\pi^n}\times\\
    &\sum_{i_2,\dots,i_{m}\geq 0}\frac{(-z_1)^{p_1}(-z_2)^{p_2+i_2d}\cdots (-z_m)^{p_m+i_md}\tilde{\bz}^{-\frac{\bc}{d}-\frac{p_1}{d}\ba_1-\frac{1}{d}\sum_{j=2}^m(p_j+di_j)\ba_j}}{\Gamma(1+p_2+i_2d)\cdots\Gamma(1+p_m+i_md)}\times\\
    &I_{p_1}\left(e^{\pi\ii}z_1\tilde{\bz}^{-\frac{\ba_1}{d}};\ba_1,\frac{\bc}{d}+\frac{p_1}{d}\ba_1+\frac{1}{d}\sum_{j=2}^m(p_j+di_j)\ba_j\right).
\end{split}
\end{equation}
The last series is convergent when $|{\rm arg}(e^{\pi\ii}z_1\tilde{\bz}^{-\frac{\ba_1}{d}})|<\frac{\pi}{d}$ and $|z_i^d\tilde{\bz}^{-\ba_i}|$ are small enough for any $i=2,\dots,m$ by the same method as \cite[Section 3.3]{matsubara2022global} in view of Theorem \ref{thm:MB} and Remark \ref{rem:integration contour}.
Moreover, it is readily seen that $\bz\in U_{T(\ba_1)}$ if $|z_1^{-d}\tilde{\bz}^{{\ba_1}}|$ and $|z_i^d\tilde{\bz}^{-\ba_i}|$ are small enough for any $i=2,\dots,m$.
To simplify the discussion, we may assume that ${\rm Re}(c_i)>0$ for any $i=1,\dots,n$ without loss of generality.
A pole $s$ with the positive real part of the integrand of $I_{p_1}\left(e^{\pi\ii}z_1\tilde{\bz}^{-\frac{\ba_1}{d}};\ba_1,\frac{\bc}{d}+\frac{p_1}{d}\ba_1+\frac{1}{d}\sum_{j=2}^m(p_j+di_j)\ba_j\right)$ is of the following form: there is an index $1\leq i\leq n$ such that $a_{i1}\neq 0$ and $i_1\in\Z$ so that
\begin{equation}\label{eq:pole spiral}
s=\frac{1}{a_{i1}}\left( \frac{c_{i}}{d}+\frac{1}{d}\sum_{j=2}^m(p_i+di_j)a_{ij}+i_1\right).
\end{equation}
It follows that we obtain an analytic continuation
\begin{equation}
\begin{split}
&\varphi(\gamma^{\bc}_{\bp};\bz)\\
=&\sum_{i=1;a_{i1}\neq 0}^n\sum_{(i_1,\dots,i_m)\in\mathbb{Z}_{\geq 0}^m}C(\ba_1,i,\bp,(i_1,p_2+i_2d,\dots,p_m+i_md),\bc)
\frac{z^{\bg_{(i_1,p_2+i_2d,\dots,p_m+i_md)}^{\bc}(\ba_1,i)}}{\Gamma(1+\bg_{(i_1,p_2+i_2d,\dots,p_m+i_md)}^{\bc}(\ba_1,i))}\\
=&\sum_{i=1;a_{i1}\neq 0}^n\sum_{r=1}^{r_i}C(\ba_1,i,\bp,\bq_r(\bp,i),\bc)
\sum_{\bu\in L_A}
\frac{z^{\bg_{\bq_r(\bp,i)}^{\bc}(\ba_1,i)+\bu}}{\Gamma(1+\bg_{\bq_r(\bp,i)}^{\bc}(\ba_1,i)+\bu)}.
\end{split}
\end{equation}
\end{proof}

\begin{example}
    Let $(n,d)=(6,7)$, $\ba_1=(2,1,1,1,1,1)$, $\ba_2=(1,1,2,1,1,1)$ and $F=z_1\bx^{\ba_1}+z_2\bx^{\ba_2}+z_3x_1^7+\cdots+z_8x_6^7$.
The index set $I$ defined by \eqref{eq:I} is subdivided into $343$ equivalence classes $I_{\bc}$ defined in \eqref{eq:definition:Ic}.
The only possible cardinalities of $I_{\bc}$ are $15,16,20,21,26$, or $30$, each of which has $48$, $72$, $108$, $72$, $42$, $1$ kinds respectively.
Let us choose $\bc=(1,2,1,1,1,1)$ so that the corresponding set $P_{\bc}$ has cardinality $26$.
To each $\bp\in P_{\bc}$, we can compute the set $Q_{\bp}:=\{\bq_r(\bp,i);r=1,\dots,r_i\}_{i=1,\dots,6}$.
For example, if $\bp=(0,0)$, $r_1=2$, $\bq_1(\bp,1)=(0,0)$, $\bq_2(\bp,1)=(1,0)$ and for $i=2,\dots,6$, $r_i=1$, $\bq_1(\bp,i)=(0,0)$.
When $\bc$ is perturbed to $\bc(\varepsilon)$, the vectors $\{\bg^{\bc(\varepsilon)}_{\bq};\bq\in Q_{\bp}\}$ are distinct, while some of them coincide when $\varepsilon=0$ so that $\bc(0)=\bc$.
More precisely, we have the following formulas:
\[
\begin{split}
    \bg_{(0,0)}^{\bc}(\ba_1,1)&=\left(-\frac{1}{2},0,0,-\frac{3}{14},-\frac{1}{14},-\frac{1}{14},-\frac{1}{14},-\frac{1}{14}\right)\\
    \bg_{(1,0)}^{\bc}(\ba_1,1)&=\left(-4,0,1,\frac{2}{7},\frac{3}{7},\frac{3}{7},\frac{3}{7},\frac{3}{7}\right)\\
    \bg_{(0,0)}^{\bc}(\ba_1,2)&=\left(-2,0,\frac{3}{7},0,\frac{1}{7},\frac{1}{7},\frac{1}{7},\frac{1}{7}\right)\\
    \bg_{(0,0)}^{\bc}(\ba_1,i)&=\left(-1,0,\frac{1}{7},-\frac{1}{7},0,0,0,0\right)\quad (i=3,4,5,6).
\end{split}
\]
It is readily seen that for a suitable choice of a weight vector $\bw\in\R^8$ such that $S(\bw)=T(\ba_1)$, the vector $\bg_{(0,0)}^{\bc}(\ba_1,i)$ for $i\geq 3$ is the leading exponent of the SST-limit \eqref{eq:basis of perturabation of gDwork}.
A straightforward computation employing Mathematica shows that the sum corresponding to the leading term is computed as
\[
\lim_{\varepsilon\to 0}\sum_{i=3}^6C(\ba_1,i,\bp,(0,0),\bc(\varepsilon))\frac{z^{\bg_{(0,0)}^{\bc(\varepsilon)}(\ba_1,i)}}{\Gamma(1+\bg_{(0,0)}^{\bc(\varepsilon)}(\ba_1,i))}
=C_0\zeta+C_1\zeta\log\zeta+C_2\zeta(\log\zeta)^2,\ \zeta=z_1^{-1}z_3^{\frac{1}{7}}z_4^{-\frac{1}{7}}.
\]
As in \eqref{eq:C_0}, $C_0$ contains numbers $\Gamma(\frac{1}{7}),\Gamma(\frac{6}{7}),\gamma,\psi^{(0)}(\frac{6}{7}),\psi^{(1)}(\frac{1}{7})$ and belong to the ring $\PG(14)$.  
\end{example}

\bibliographystyle{siam}
\bibliography{references.bib}

\bigskip
\noindent
\begin{tabular}{ll}
\textbf{Saiei-Jaeyeong Matsubara-Heo} & \textbf{Masanori Asakura} \\
Graduate School of Information Sciences & Department of Mathematics \\
Tohoku University & Hokkaido University \\
Sendai 980-0845, Japan & Sapporo 060-0808, Japan \\
\texttt{saiei@tohoku.ac.jp} & \texttt{asakura@math.sci.hokudai.ac.jp}
\end{tabular}

\end{document}